\documentclass[psamsfonts,fceqn,leqno]{amsart}
  \usepackage{amsthm}
  \usepackage{amsmath}
  \usepackage{amssymb}
  \usepackage{mathrsfs}
  \usepackage{multirow}
  
  \newtheorem{theorem}{Theorem}[section]
  \newtheorem{corollary}[theorem]{Corollary}
  \newtheorem{proposition}[theorem]{Proposition}
  \newtheorem{lemma}[theorem]{Lemma}

  \theoremstyle{definition}
  \newtheorem{definition}[theorem]{Definition}
  \newtheorem{remark}[theorem]{Remark}

  \numberwithin{equation}{section}

  \title[]
  {On the core and Walrasian expectations\\ 
  equilibrium in infinite dimensional\\ 
  commodity spaces}
  \author[A. Bhowmik]{Anuj Bhowmik}
  \address{School of Computing and Mathematical Sciences,
  Auckland University of Technology, Private Bag 92006, Auckland
  1142, New Zealand}
  \email{anuj.bhowmik@aut.ac.nz}
  \author[J. Cao]{Jiling Cao}
  \address{School of Computing and Mathematical Sciences,
  Auckland University of Technology, Private Bag 92006, Auckland
  1142, New Zealand}
  \email{jiling.cao@aut.ac.nz}
  
  \thanks{\hspace{-1.66em} \emph{JEL Classification Numbers:}
  D51, D82, D11.}

  \date{}

\begin{document}

  \begin{abstract}
  In this paper, we establish two different characterizations
  of Walrasian expectations allocations by the veto power of
  the grand coalition in an asymmetric information economy
  having finite numbers of agents and states of nature and
  whose commodity space is a Banach lattice. The first one deals
  with Aubin non-dominated allocations, and the other claims
  that an allocation is a Walrasian expectations allocation if
  and only if it is not privately dominated by the grand
  coalition, by considering perturbations of the original
  initial endowments in precise directions.
  \end{abstract}

  \maketitle

  \section{Introduction}

  The classical deterministic Arrow-Debreu-McKenzie model on an
  economic system consists of finitely many consumers, producers
  and commodities, refer to \cite{Arrow-Debreu:54} and 
  \cite{McKenzie:59}. In late of 1950's, Arrow and Debreu 
  introduced uncertainty into this deterministic model by adding 
  contingent claims, \cite[Chapter 7]{Debreu:59}. In this 
  improved model, agents make contracts contingent on the 
  realized state of nature known to all the agents. However, 
  such a model does not capture the idea of contracts under 
  asymmetric information as all agents face the same uncertainty. 
  To overcome this shortcoming, Radner \cite{Radner:68} 
  introduced economies with asymmetric information. In
  Radner's model, an economy consists of finitely many agents,
  each of whom is characterized by a state dependent utility
  function, a random initial endowment, a private information set
  and his prior belief; and agents arrange contingent contracts for
  trading commodities before they obtain any information about the
  realized state of nature. Analogous to the concept of a
  Walrasian equilibrium in the Arrow-Debreu-McKenzie model, 
  Radner also introduced the notion of a Walrasian expectations
  equilibrium for an asymmetric information economy so that the
  information of an agent places a restriction on his feasible
  trades: better information allows for more contingent trades,
  and each agent maximizes his ex ante expected utility subject
  to his budget constraint with respect to his private information.
  This individualistic behavior leads to a feasible redistribution
  of the initial endowments for each state of nature.

  In this paper, we continue the study on asymmetric information
  economies. In particular, we are interested in asymmetric
  information economies whose commodity spaces are infinite
  dimensional spaces. In Section \ref{sec:model}, we give a
  general description on a discrete model of an asymmetric
  information economy having finitely many of agents and states
  of nature, and whose commodity space is a Banach lattice. We
  also associate a continuum model with this discrete model.
  The continuum model has the equal treatment property. It is
  worth to mention that Tourky and Yannelis \cite{tourky:01} 
  and Podczeck \cite{Podczeck:03} constructed counterexamples 
  of economies to show that the classical core-Walras 
  equivalence in \cite{Aumann:64} fails whenever the commodity 
  space is a non-separable ordered Banach space. In both of 
  these two papers, the authors used the
  Continuum Hypothesis in set theory to construct an economy with
  uncountably many utility functions. Since the commodity spaces
  of economies in our paper are Banach lattices which are not
  necessarily separable, our theorems give some positive results
  for asymmetric information economies whose commodity spaces
  are non-separable.

  The concept of the private core of an asymmetric information
  economy was introduced in \cite{Yannelis:91}. Einy et al. 
  \cite{Einy-Moreno-Shitovitz:01}
  showed that under appropriate assumptions, the private core
  coincides with the set of Walrasian expectations allocations
  for an atomless economy whose commodity space is Euclidean
  space. Herv\'{e}s-Beloso et al. 
  \cite{Herves-Beloso-Moreno-Garcia-Yannelis:05b} established 
  a similar
  result for an equal treatment economy whose commodity space
  is $\ell_\infty$. By considering an atomless economy with
  asymmetric information, Evren and H\"{u}sseinov 
  \cite{Evern-Husseinov:08}
  extended these results to an economy whose commodity space
  is an ordered separable Banach space which has an interior
  point in its positive cone. In Section \ref{sec:privatecore},
  we establish a similar equivalence result for an equal
  treatment economy whose commodity space
  is a Banach lattice. In an atomless economy with a complete
  information and a finite dimensional commodity space, improving
  a result of \cite{Aumann:64}, Schmeidler \cite{Schmeidler:72} 
  and Grodal \cite{Global:72}
  showed that (i) if some coalition blocks an allocation, then
  there is also a blocking subcoalition with arbitrarily small
  measure; and (ii) the small coalition can be the union of at
  most $\ell$ +1 coalitions, each of which has measure and
  diameter less than an arbitrary small number $\epsilon> 0$,
  where $\ell$ is the number of commodities. In the same issue
  of Econometrica, Vind \cite{Vind:72} extended (i) to show that 
  if some
  coalition blocks then there is a blocking coalition with any
  measure less than the measure of the grand coalition. The first
  extension of the Schmeidler and Grodal's results to an infinite
  dimensional setting (the space $\ell_\infty$) were obtained in
  \cite{Herves-Moreno-Nunez-Pascoa:00}, where it was also showed
  that the Vind's result fails in the space $\ell_\infty$ under
  the standard assumptions. Herv\'{e}s-Beloso et al. 
  \cite{Herves-Beloso-Moreno-Garcia-Yannelis:05a}
  first extended Vind's theorem to an asymmetric information
  economy with a continuum of agents having the equal treatment
  property, and whose commodity space is Euclidean space. Later,
  Herv\'{e}s-Beloso et al. 
  \cite{Herves-Beloso-Moreno-Garcia-Yannelis:05b} extended Vind's 
  theorem to
  an asymmetric information economy with a continuum of agents
  having the equal treatment property and whose commodity space
  is $\ell_\infty$. All of these extension are established by
  using finite-dimensional Lyapunov's convexity theorem. Using
  an infinite dimensional extension of Lyapunov's convexity
  theorem, Evren and H\"{u}sseinov \cite{Evern-Husseinov:08} 
  further extended
  Vind's theorem to an atomless economy whose commodity space
  is an ordered Banach space having an interior point in its
  positive cone, with some additional assumption. Our second
  main result in Section \ref{sec:privatecore} is an extension
  of Vind's theorem to an asymmetric information economy with
  the equal treatment property, and whose commodity space is a
  Banach lattice.

  Addressing complete information economies with finitely many
  agents and commodities, Aubin \cite{Aubin:79} introduced the 
  ponder veto concept and showed that the core obtained by this 
  veto mechanism coincides with the Walrasian equilibria. 
  Aubin's approach was employed by Gabriella Graziano, Meo and
  Herv\'{e}s-Beloso et al. to characterize Walrasian
  expectations equilibria of asymmetric information economies.
  Herv\'{e}s-Beloso et al. \cite{Herves-Beloso-Moreno-Garcia-Yannelis:05a} 
  introduced the notion of
  Aubin non-dominated allocations. Gabriella Graziano and
  Meo \cite{Graziano-Meo:05} showed that the Aubin private core provides a
  complete characterization of Walrasian expectations
  allocations in an economy with a complete measure space of
  agents. They also showed that Aubin private core allocations
  in an asymmetric information economy with a complete measure
  space of agents and an ordered separable
  Banach space whose positive cone has an interior point as
  the commodity space can be characterized by privately
  non-dominated allocations in suitable associated economies.
  Herv\'{e}s-Beloso et al. 
  \cite{Herves-Beloso-Moreno-Garcia-Yannelis:05a, Herves-Beloso-Moreno-Garcia-Yannelis:05b} established similar
  results for asymmetric information economies with finitely
  many agents and states of nature, and ${\mathbb R}^\ell$ or
  $\ell_\infty$ as the commodity space.
  The proof of Herv\'{e}s-Beloso et al. was built on associated
  continuum economies and thus requires the validity of Vind's
  theorem, which is not the case in the proof of Gabriella
  Graziano and Meo \cite{Graziano-Meo:05}. In Section \ref{sec:Equitheodiseco},
  we adopt the approach in \cite{Herves-Beloso-Moreno-Garcia-Yannelis:05a, Herves-Beloso-Moreno-Garcia-Yannelis:05b}
  to characterize Walrasian expectations equilibria in terms of
  the private blocking power of the grand coalition. Since our
  economic model has finitely many agents and a Banach lattice
  as the commodity space, our results can be considered as
  extensions of those in \cite{Herves-Beloso-Moreno-Garcia-Yannelis:05a, Herves-Beloso-Moreno-Garcia-Yannelis:05b}.

  In the Appendix, an asymmetric information economy whose commodity
  space is a Banach lattice and which has infinitely many states
  of nature is discussed. Mathematical preliminaries and discussions
  on some assumptions are also provided in the Appendix. Of course,
  all of these mathematical preliminaries can be found in 
  \cite{Aliprantis-Border:05}.

  \section{Description of the model} \label{sec:model}

  In this section, we describe our model of an exchange economy
  with asymmetric information.

  \subsection{\it The discrete model}
  Our first model of an exchange economy is a discrete model
  $\mathcal E$ with (fixed) $n$ agents denoted by the set $N=\{1,
  ..., n\}$, like those ones considered in \cite{Radner:68, 
  Radner:82}. A measurable
  space $(\Omega, \mathcal{F})$ is used to describe the exogenous
  uncertainty of $\mathcal{E}$, where $\Omega$ is a finite set
  denoting the set of all possible states of nature
  and the $\sigma$-algebra $\mathcal{F}$ denotes the set of all
  events. The commodity space of $\mathcal{E}$ is a Banach lattice
  $Y$ with a partial order $\leq$. The economy $\mathcal{E}$ extends
  over two time periods $\tau = 0, 1$, and consumption takes place
  at $\tau= 1$. At $\tau= 0$, there is uncertainty over the states
  of nature and agents make contracts that may be contingent on the
  realized state of nature at $\tau= 1$. More precisely, 
  $\mathcal{E}$ is expressed by
  ${\mathcal E} = \{((\Omega, \mathcal{F}), Y_+,
  \mathcal{F}_{i}, U_{i}, a_{i}, q_i): i\in N\}$,
  where the positive cone $Y_+$ of $Y$ is the \emph{consumption set}
  in each state of nature $\omega\in \Omega$ for every agent $i\in N$;
  ${\mathcal F}_i$ is a partition of $\Omega$ representing the
  \emph{private information} of agent $i$; $U_{i}:\Omega \times
  Y_+ \to \mathbb R$ is the \emph{random utility function} of
  agent $i$; $a_{i}: \Omega\rightarrow Y_+$ is the \emph{random
  initial endowment} of agent $i$, assumed to be constant on
  elements of ${\mathcal F}_i$; and $q_i$ is a probability
  measure on $\Omega$ giving the \emph{prior} of agent $i$,
  assumed to be positive on all elements of $\Omega$. For any
  random consumption bundle $x:\Omega\rightarrow Y_+$, the
  \emph{ex ante expected utility} of agent $i$ is given by
  $V_{i}(x) = \sum_{\omega\in \Omega}U_{i} (\omega,
  x(\omega))q_i(\omega)$.

  An \emph{assignment} in $\mathcal{E}$ is a function $x= (x_1,
  ..., x_n)$ which associates to every agent $i$ a random
  consumption bundle $x_i: \Omega \to Y_+$, equivalently written
  as $x_i \in (Y_+)^\Omega$. We call a function
  with domain $\Omega$, constant on elements of $\mathcal{F}_{i}$,
  \emph{$\mathcal{F}_{i}$-measurable}, although measurability
  is meant with respect to the $\sigma$-algebra generated by the
  partition. Put
  $L_i = \{ x_{i} \in (Y_+)^\Omega: x_{i} \mbox{ is }
  \mathcal{F}_{i}\mbox{-measurable}\}$.
  An assignment $x= (x_1, ..., x_n)$ in $\mathcal{E}$ is called
  an \emph{allocation} if $x_i\in L_i$ for all $i\in N.$
  An allocation $x$ is said to be \emph{feasible} if
  $\sum_{i= 1}^{n}x_{i}(\omega)\leq \sum_{i= 1}^{n}a_{i}(\omega)$
  for all $\omega \in \Omega$. Further, an allocation $x$ of
  $\mathcal{E}$ is \emph{privately dominated} if there exists
  a feasible allocation $y= (y_1,...,y_n)$ such that $V_i(y_i)>
  V_i(x_i)$ for all $i\in N$. A feasible allocation $x$ of
  $\mathcal E$ is called \emph{privately Pareto optimal} if it
  is not privately dominated. A \emph{price system} is an
  $\mathcal{F}$-measurable, non-zero function $\pi:\Omega \to
  Y_+^\ast$, where $Y_+^\ast$ is the positive cone of the
  norm-dual space $Y^\ast$ of $Y$. The \emph{budget set} $B_{i}
  (\pi)$ of agent $i$ under $\pi$ is defined as
  \[
  B_{i}(\pi) = \left\{ x_{i}\in L_i:\sum_{\omega\in
  \Omega} \langle\pi(\omega), x_{i}(\omega)\rangle\leq
  \sum_{\omega\in \Omega}\langle\pi(\omega), a_{i}(\omega)
  \rangle\right\}.
  \]
  Given a feasible allocation $x$ and a price system $\pi$ in
  $\mathcal E$, the pair $(x, \pi)$ is called a \emph{Walrasian
  expectations quasi-equilibrium} of $\mathcal{E}$ in the sense
  of Radner if
  \begin{itemize}
  \item[(2.1)] for all $i\in N$, $x_i \in B_i(\pi)$ and the
  consumption bundle $x_i$ maximizes $V_i$ on $B_i(\pi)$
  whenever $\sum_{\omega\in \Omega}\langle\pi(\omega),
  a_i(\omega)\rangle \not= 0$;
  \item[(2.2)]
  $\sum_{\omega\in \Omega} \left \langle\pi(\omega),\
  \sum_{i\in N} x_i(\omega)\right\rangle=
  \sum_{\omega\in \Omega} \left \langle\pi(\omega),\
  \sum_{i\in N} a_i(\omega)\right\rangle$.
  \end{itemize}
  If $\sum_{\omega \in \Omega}\langle\pi(\omega),
  a_i(\omega)\rangle\neq 0$ for some $i\in N$, then $(x, \pi)$ is
  called \emph{non-trivial}. Further, if
  \begin{itemize}
  \item[(2.1$^\prime$)] for all $i\in N$, $x_i \in B_i(\pi)$ and
  $x_i$ maximizes $V_i$ on $B_i(\pi)$,
  \end{itemize}
  and (2.2) hold, $(x, \pi)$ is called a \emph{Walrasian expectations
  equilibrium} of $\mathcal{E}$ in the sense of Radner, and
  $x$ is called a \emph{Walrasian expectations allocation}. Note that
  if $(x, \pi)$ is a Walrasian expectations quasi-equilibrium, then
  feasibility of $x$ and (2.2) together imply
  \begin{itemize}
  \item[(2.3)] $\left \langle\pi(\omega),\
  \sum_{i\in N} x_i(\omega)\right\rangle=
   \left \langle\pi(\omega),\
  \sum_{i\in N} a_i(\omega)\right\rangle$ for each $\omega\in \Omega$.
  \end{itemize}

  Throughout this paper, we put the following assumptions on our
  discrete model $\mathcal E$. These assumptions or some of their
  combinations are used in different places of the paper.

  \begin{itemize}
  \item[(A$_1$)] \emph{Continuity}. For each $i\in N$ and $\omega
  \in\Omega$, $U_{i}(\omega, \cdot)$ is norm-continuous.

  \item[(A$_2$)] \emph{Monotonicity}. For each $i\in N$
  and $\omega\in \Omega$, $U_{i}(\omega, \cdot)$ is monotone in the
  sense that if $x, y\in Y_+$ with $y> 0,$ then $U_i(\omega, x+y)>
  U_i(\omega, x)$.

  \item[(A$_3$)] \emph{Concavity}. For each $i\in N$ and $\omega\in
  \Omega$, $U_{i}(\omega, \cdot)$ is concave.

  \item[(A$_4$)] \emph{Quasi-interiority}. For each $\omega\in \Omega$,
  $\sum_{i\in N}a_i(\omega) \gg 0$.

  \item[(A$_4^\prime$)] \emph{Positivity}. For each $\omega\in \Omega$,
  $\sum_{i\in N}a_i(\omega)> 0$.

  \item[(A$_5$)] \emph{Strong positivity}. For each $i\in N$,
  $\inf\{a_i(\omega): \omega\in \Omega\}> 0$.

  \item[(A$_6$)] \emph{Stability}. There exists an element $\hat{a}
  \in Y_+$ such that $L\left(\sum_{i\in N}a_i(\omega)\right)=
  L(\hat{a})$ for each $\omega\in \Omega$.

  \item[(A$_7$)] \emph{Irreducibility}. For each feasible allocation
  $x$ of $\mathcal E$ and any two non-empty disjoint subsets $N_1,
  N_2$ of $N$ with $N_1 \cup N_2 =N$, there is a $y=(y_i)_{i\in N_2}$
  such that $y_i\in L_i$ and $V_i(y_i)> V_i(x_i)$ for all $i\in N_2$,
  and $\sum_{i\in N_1} a_i(\omega)+ \sum_{i\in N_2}x_i(\omega)\geq
  \sum_{i\in N_2} y_i(\omega)$ for each $\omega\in \Omega$.
  \end{itemize}

  \begin{remark}
  Under (A$_1$), $V_{i}$ is continuous with respect to the product
  topology induced by the norm. Under (A$_2$), $V_i$
  is monotone in the sense that if $x, y\in (Y_+)^\Omega$ with
  $y(\omega)> 0$ for some $\omega\in \Omega$, then $V_i(x+y)> V_i(x)$.
  Under (A$_3$), $V_{i}$ is concave. (A$_6$) implies that the
  aggregate endowments do not vary too much across states. (A$_7$)
  is similar to ($A.3$) in \cite{Einy-Moreno-Shitovitz:01}. 
  When $Y=\ell_\infty$, assumptions (A$_2$)-(A$_4$) are similar 
  to (A.2)-(A.4) in \cite{Herves-Beloso-Moreno-Garcia-Yannelis:05b}. 
  Further, assumptions (A$_1$)-(A$_4$) are similar to (A.2) and 
  (A.3) in \cite{Graziano-Meo:05}, and (A.A3)-(A.A5) in 
  \cite{Evern-Husseinov:08} (except for the convexity). \qed
  \end{remark}

  The following proposition, which is a modification of Proposition
  $3.2$ in \cite{Einy-Moreno-Shitovitz:01}, shall be used in the sequel.

  \begin{proposition}\label{prop:Irreducibilty}
  If $\mathcal E$ satisfies \emph{(A$_7$)}, then every
  non-trivial Walrasian expectations quasi-equilibrium of
  $\mathcal E$ is a Walrasian expectations equilibrium.
  \end{proposition}
  \begin{proof}
  Let $(x, \pi)$ be a non-trivial Walrasian expectations
  quasi-equilibrium of $\mathcal E$.
  Then $N_2= \{i\in N: \sum_{\omega\in \Omega}\langle \pi(\omega),
  a_i(\omega)\rangle\neq 0\} \neq \emptyset$. Let $N_1=N- N_2$.
  If $N_2= N$, there is nothing to verify. Otherwise, $N_1\neq
  \emptyset$. By (A$_7$), there is a $y=(y_i)_{i\in N_2}$ such
  that $y_i\in L_i$ and $V_i(y_i)> V_i(x_i)$ for all $i\in N_2$,
  and $\sum_{i\in N_1}
  a_i(\omega)+\sum_{i\in N_2} x_i(\omega)\geq \sum_{i\in N_2}
  y_i(\omega)$ for all $\omega\in \Omega$. Then $\sum_{i\in N_2}
  \sum_{\omega\in \Omega}\langle\pi(\omega),
  x_i(\omega)\rangle\geq \sum_{i\in N_2}\sum_{\omega\in \Omega}
  \langle\pi(\omega), y_i(\omega)\rangle$, and
  $\sum_{\omega\in \Omega}\langle\pi(\omega), x_i(\omega)\rangle<
  \sum_{\omega\in \Omega}\langle\pi(\omega), y_i(\omega)\rangle$
  for each $i\in N_2$, which is a contradiction.  
  \end{proof}

  \subsection{\it A continuum interpretation}

  Now, we associate a continuum economy $\mathcal{E}_c$ with the
  discrete model $\mathcal{E}$, just like that in
  \cite{García-Cutrín-Hervés-Beloso:93} and 
  \cite{Herves-Beloso-Moreno-Garcia-Yannelis:05a,
  Herves-Beloso-Moreno-Garcia-Yannelis:05b}. The set of agents of
  $\mathcal{E}_c$ is the unit interval $I= [0, 1]$ endowed with
  the Lebesgue measure $\mu$. We write $I =\bigcup_{i=1}^n I_i,$
  where $I_i=\left[\frac{i-1}{n}, \frac{i}{n}\right)$ for $i=
  1,...,n-1$, and $I_n=\left[\frac{n-1}{n}, 1\right]$. Each agent
  $t\in I_i$ is characterized by the private information set
  $\mathcal{F}_t=\mathcal{F}_i$; the consumption set $Y_+$ in
  each state $\omega\in \Omega$; the random initial endowment
  $a(t,\cdot)=a_i$; the random utility function $U_t = U_i$;
  and the prior $q_t= q_i$. The ex ante expected utility
  function of every agent $t\in I_i$ is $V_t=V_i$. Thus,
  $\mathcal{E}_c$ can be expressed as
  $\mathcal{E}_c =\{((\Omega, \mathcal{F}), Y_+, I,
  \mathcal{F}_i,
  a_i, V_i,q_i): i\in N\}$.
  We call $I_i$ the set of \emph{type $i$ agents}, and
  $\mathcal{E}_c$ an economy with the
  \emph{equal treatment property}. An \emph{assignment}
  in $\mathcal{E}_c$ is a function $f:I \times \Omega\rightarrow
  Y_+$ such that $f(\cdot,\omega) \in L_1 (\mu, Y_+)$ for all
  $\omega\in \Omega$, where $L_1 (\mu, Y_+)$ is the set of all
  Bochner integrable functions from $I$ into $Y_+$. Put $L_t =
  L_i$ for each $t \in I_i$ and $i\in N$. An
  assignment $f$ in $\mathcal{E}_c$ is called an \emph{allocation}
  if $f(t,\cdot)\in L_t$ for almost all $t\in I$. An allocation
  $f$ in $\mathcal{E}_c$ is {\it feasible} if $\int_I f(t, \omega)
  d\mu(t) \leq \int_I a(t, \omega) d\mu(t)$ for all $\omega\in
  \Omega$. A \emph{coalition} in $\mathcal{E}_c$ is a Borel
  measurable subset $S \subseteq I$ with $\mu(S)> 0.$ A coalition
  $S$ \emph{privately blocks an allocation} $f$ in $\mathcal{E}_c$
  if there is a function $g: S\times \Omega\rightarrow Y_+$ such that
  $g(t, \cdot) \in L_t$ and  $V_t(g(t,\cdot))> V_t(f(t,\cdot))$
  for all $t\in S$, and $\int_{S}g(t, \omega)d\mu(t)\leq \int_{S}
  a(t, \omega)d\mu(t)$ for all $\omega\in \Omega$. The \emph{private
  core} of $\mathcal{E}_c$ is the set of all feasible allocations
  which are not privately blocked by any coalition.

  Given a feasible allocation $f$ and a price system $\pi$ in
  $\mathcal{E}_c$, the budget set of an agent $t\in I$ is
  $B_t(\pi)=B_i(\pi)$ if $t\in I_i$ and $i\in N$. The pair
  $(f, \pi)$ is called a \emph{Walrasian expectations
  quasi-equilibrium} of $\mathcal{E}_c$ in the sense of Radner
  if
  \begin{enumerate}
  \item[(2.4)] for all $t\in I$,  $f(t,\cdot) \in B_t(\pi)$ and
  $f(t,\cdot)$ maximizes $V_t$ on $B_t(\pi)$ whenever
  $\sum_{\omega\in \Omega}\langle \pi(\omega), a(t, \omega)
  \rangle\not= 0$;
  \item[(2.5)] $\sum_{\omega\in \Omega} \left\langle\pi(\omega),
  \int_I f(t, \omega)d\mu(t) \right\rangle= \sum_{\omega\in
  \Omega} \left\langle\pi(\omega), \int_I a(t, \omega)d\mu(t)
  \right\rangle$.
  \end{enumerate}
  If $\sum_{\omega\in\Omega}\langle\pi(\omega), a(t, \omega)
  \rangle\neq 0$ for all $t$ in some coalition $S\subseteq I$,
  then $(f, \pi)$ is called \emph{non-trivial}. Further, if
  \begin{enumerate}
  \item[(2.4$^\prime$)] for all $t\in I$, $f(t,\cdot)
  \in B_t(\pi)$ and $f(t, \cdot)$ maximizes $V_t$ on $B_t(\pi)$
  \end{enumerate}
  and (2.5) hold, then $(f, \pi)$ is called a \emph{Walrasian
  expectations equilibrium} of $\mathcal{E}_c$ in the sense of
  Radner, and $f$ is called a \emph{Walrasian expectations
  allocation}. An allocation $f$ in $\mathcal{E}_c$ can be
  interpreted as an allocation $x$ in $\mathcal{E},$ where
  $x_i =n\int_{I_i}f(t, \cdot)d\mu(t)$ for all $i\in N$.
  Conversely, an allocation $x$ in $\mathcal{E}$ can be
  interpreted as an allocation $f$ in $\mathcal{E}_c$, where
  $f$ is the simple function given by $f(t, \cdot)= x_i$
  for all $t\in I_i$ and $i\in N$.

  Analogous to Theorem 1 in \cite{García-Cutrín-Hervés-Beloso:93}
  and Theorem 3.1 in  \cite{Herves-Beloso-Moreno-Garcia-Yannelis:05a,
  Herves-Beloso-Moreno-Garcia-Yannelis:05b},
  our next result shows that the discrete and continuum
  approaches can be considered equivalent with respect to
  Walrasian expectations (quasi-)equilibria.

  \begin{proposition}\label{prop:Equivalency}
  Assume that $\mathcal{E}$ satisfies \emph{(A$_3$)}. If
  $(x,\pi)$ is a non-trivial Walrasian
  expectations quasi-equilibrium of $\mathcal{E}$, then $(f,
  \pi)$ is a non-trivial Walrasian expectations quasi-equilibrium
  of $\mathcal{E}_c$, where $f(t,\cdot)= x_i$ if $t\in I_i$.
  Conversely, if $(f,\pi)$ is a non-trivial Walrasian
  expectations quasi-equilibrium of $\mathcal{E}_c$, then $(x,
  \pi)$ is a non-trivial Walrasian expectations quasi-equilibrium
  of $\mathcal{E}$, where $x_i= n \int_{I_i}f(t,\cdot) d\mu(t)$.
  \end{proposition}

  Since the proof of Proposition \ref{prop:Equivalency} is
  straightforward, we omit it.

  \begin{remark} \label{rem:discont}
  A similar conclusion holds if ``non-trivial Walrasian expectations
  quasi-equilibrium" is replaced with ``Walrasian expectations
  equilibrium". \qed
  \end{remark}

  \section{Characterizations of private cores of equal
  treatment economies} \label{sec:privatecore}

  In this section, we establish a relation between the private
  core and the set of Walrasian expectations allocations in the
  setting of equal treatment, and give an extension of Vind's
  theorem. These two results allow us to obtain our main
  theorems in Section \ref{rem:fuzzycoreequilibria2}.

  \subsection {\it Equivalence results} \label{sec:private core}

  Evren and H\"{u}sseinov \cite{Evern-Husseinov:08} provided 
  an equivalence theorem
  between the private core and the set of Walrasian
  expectations allocations in an economy whose commodity space is
  an ordered separable Banach space having an interior point in
  its positive cone. Next, we give a similar result for the case
  that the commodity space is a Banach lattice.

  \begin{theorem}\label{thm:core-wal1}
  Assume that the commodity space of $\mathcal{E}$ has an
  interior point in its positive cone. Let $f$ be a feasible
  allocation in $\mathcal{E}_c$ such that $f(t,\cdot)= x_i$
  for all $t\in I_i$ and $i\in N$. Under \emph{(A$_1$), (A$_2$)}
  and \emph{(A$_4$)}, if $f$ is in the private core of
  $\mathcal{E}_c$, then $(f, \pi)$ is a non-trivial Walrasian
  expectations quasi-equilibrium of $\mathcal{E}_c$ for some
  non-zero $\pi: \Omega \to Y^\ast_+$.
  \end{theorem}

  \begin{proof}
  Consider a correspondence $G: I\rightrightarrows (Y_+)^\Omega$
  defined by $G(t) = \{ g(t, \cdot)\in L_t: V_t(g(t, \cdot))>
  V_t(f(t,\cdot))\}$. By (A$_2$), $G(t) \ne \emptyset$ for all
  $t\in I$. Applying the infinite dimensional extension of
  Lyapunov convexity theorem \cite{uhl:69} to the proof
  of Proposition 5 in \cite[p. 62]{Hildenbrand:74}, one can show
  that
  \[
  H= \|\cdot\|^\Omega\mbox{-}{\rm cl}\left(\bigcup \left
  \{\int_A G(t,\cdot)d\mu(t)-\int_A a(t, \cdot)d\mu(t): A\in
  \mathcal{A},\ \mu(A)> 0\right\}\right)
  \]
  is a convex subset of $Y^\Omega$, where
  $\mathcal{A}$ denotes the set of Lebesgue measurable subsets
  of $I$. Since $H\cap {\rm int} (-Y_+)^\Omega= \emptyset$,
  by the separation theorem, there is a non-zero positive element
  $\pi\in (Y^\ast)^\Omega$ such that for any coalition $A$,
  \[
  \sum_{\omega\in \Omega}\langle \pi(\omega), y(\omega)\rangle
  \ge \sum_{\omega\in \Omega}\left\langle \pi(\omega), \int_A
  a(t,\omega) d\mu(t) \right\rangle
  \]
  for all $y\in \int_A G(t, \cdot)d\mu(t)$. Let $N_1=\{i\in N:
  \sum_{\omega\in \Omega}\langle \pi(\omega), a_i(\omega)
  \rangle\neq 0\}$. Suppose $V_i(y_i)>V_i(x_i)$ for some $i\in
  N_1$ and $y_i\in L_i$. If $y_i \in B_i(\pi)$, by (A$_1$), one
  can construct some $z_i \in B_i(\pi)$ such that $V_i(z_i)>
  V_i(x_i)$ and
  \[
  \sum_{
  \omega \in \Omega}\left\langle \pi(\omega), \int_{I_i}z_i(\omega)
  d\mu(t) \right\rangle< \sum_{\omega\in \Omega}\left\langle
  \pi(\omega), \int_{I_i}a_i(\omega)d\mu(t)\right\rangle,
  \]
  which is a contradiction. Thus, $\sum_{\omega\in
  \Omega}\langle\pi(\omega), y_i(\omega)\rangle> \sum_{\omega
  \in \Omega}\langle\pi(\omega), a_i(\omega)\rangle$.
  By (A$_2$), $\sum_{\omega
  \in \Omega}\langle \pi(\omega), x_i(\omega)\rangle\geq \sum_{
  \omega\in \Omega}\langle \pi(\omega), a_i(\omega)\rangle$ for
  all $i\in N_1$. Using the feasibility of $f$, one can show
  that $x_i \in B_i(\pi)$ for all $i\in N_1$. Thus, $(x, \pi)$
  is a non-trivial Walrasian expectations quasi-equilibrium in
  $\mathcal E$. By Proposition \ref{prop:Equivalency}, $(f,\pi)$
  is a non-trivial Walrasian expectations quasi-equilibrium in
  $\mathcal{E}_c$. 
  \end{proof}

  \begin{corollary} \label{coro:quasiWalras}
  Let $f$ be a feasible allocation such that $f(t, \cdot)= x_i$
  for all $t\in I_i$ and $i\in N$.
  Under {\rm (A$_1$), (A$_2$)}, {\rm (A$_4$)} and {\rm (A$_7$)},
  $f$ is a Walrasian expectations allocation if and only if $f$
  is in the private core of $\mathcal{E}_c$.
  \end{corollary}

  Next, we extend Theorem \ref{thm:core-wal1} to an asymmetric
  information economy with equal treatment property whose commodity
  space is a Banach lattice containing a quasi-interior point in
  its positive cone. For each $i\in N$ and each $x_i\in L_i$,
  let $P_i(x_i):=\{y_i\in L_i: V_i(y_i)> V_i(x_i)\}$ be the set
  of all $\mathcal{F}_i$-measurable consumption bundles preferred
  to $x_i$ by agent $i$. Then $P_i: L_i \rightrightarrows L_i$
  is called the \emph{preference relation} of agent $i$. Note that
  under (A$_1$) and (A$_3$), $P_i(x_i)$ is convex and relatively
  $\|\cdot\|^\Omega$-open in $L_i$ for all $i\in N$. The following
  definition of ATY-properness is taken from \cite{Podczeck:08}.

  \begin{definition}
  The relation $P_i: L_i \rightrightarrows L_i$ is
  called \emph{ATY-proper} at
  $x_i\in L_i$ if there exists a convex subset $\widetilde{P}_i
  (x_i)$ of $Y^\Omega$ with non-empty $\|\cdot\|^\Omega$-interior
  such that $\widetilde{P}_i(x_i)\cap L_i= P_i(x_i)$ and
  ($\|\cdot\|^\Omega$-int$\widetilde{P}_i(x_i)$)$\cap L_i\neq
  \emptyset$.
  \end{definition}

  \begin{itemize}
  \item[(A$_8$)] \emph{Properness}. If $(x_1,...,x_n)$ is a
  privately Pareto optimal allocation in $\mathcal{E}$, then for
  each $i\in N$, $P_i$ is ATY-proper at $x_i$.
  \end{itemize}

  \begin{lemma}\label{lem:contfunc}
  Let $Y$ be a real vector space endowed with a Hausdorff, locally
  convex topology $\tau$ and let $U, V$ be convex subsets of
  $Y$ such that $U$ is open and $U\cap V\neq \emptyset$. Let $y\in
  V\cap {\rm cl} U$, where ${\rm cl}U$ denotes the closure of $U$.
  Suppose that $\pi$ is a linear functional (not necessarily
  continuous) on $Y$ with $\langle\pi, y\rangle\leq \langle\pi,
  y^\prime \rangle$ for all $y^\prime\in U\cap V$. Then, there exist
  linear functionals $\pi_1$ and $\pi_2$ on $Y$ such that $\pi_1$
  is continuous, $\langle \pi_1, y\rangle\leq \langle\pi_1, u
  \rangle$ for all $u\in U$, $\langle \pi_2, y\rangle\leq \langle
  \pi_2, v\rangle$ for all $v\in V$ and $\pi= \pi_1+ \pi_2$.
  \end{lemma}

  \begin{lemma}\label{lem:dense}
  Let $Y$ be a Riesz space endowed with a Hausdorff, locally
  convex topology $\tau$. If $L(z)$ is $\tau$-dense in $Y$,
  then $L(z)_+$ is $\tau$-dense in $Y_+$.
  \end{lemma}

  \begin{lemma}\label{lem:rieszdec}
  Let $Y$ be a Riesz space and let $Z$ be an ideal in $Y$. Let
  $y_1,..., y_m$ be elements of $Y$ and $z_1, ..., z_m$ be
  elements of $Z$ such that $\sum_{i= 1}^m z_i\leq \sum_{i= 1}^m
  y_i$. Suppose that there exists an element $z\in Z$ such that
  $z\leq y_i$ for each $i= 1, ..., m$. Then, there are elements
  $\hat{z}_1, ..., \hat{z}_m$ of $Z$ such that $\sum_{i= 1}^m
  \hat{z}_i= \sum_{i= 1}^m z_i$ and $\hat{z}_i\leq y_i$ for
  each $i= 1, ..., m$.
  \end{lemma}

  For proofs of Lemmas \ref{lem:contfunc}, \ref{lem:dense} and
  \ref{lem:rieszdec}, see Lemmas 2 and 3 in \cite{Podczeck:96}
  and Lemma 7 in \cite{Podczeck:08}, respectively.
  In the proof of the next theorem, the argument to get
  continuity of the equilibrium price is similar to that in
  Theorem 2 of \cite{Podczeck:08}. Our proof needs
  some additional construction because of the free disposal
  assumption.

  \begin{theorem}\label{thm:core-wal2}
  Assume that $\mathcal E$ satisfies \emph{(A$_1$)-(A$_4$)},
  \emph{(A$_6$)} and \emph{(A$_8$)}. Let $f$ be a feasible
  allocation in $\mathcal{E}_c$ such that $f(t, \cdot)= x_i$ for
  all $t\in I_i$ and $i\in N$. If $f$ is in the private core of
  $\mathcal{E}_c$, then $(f, \pi)$ is a non-trivial Walrasian
  expectations quasi-equilibrium of $\mathcal{E}_c$ for some
  non-zero $\pi: \Omega \to Y^\ast_+$.
  \end{theorem}

  \begin{proof}
  Let $f$ be in the private core of $\mathcal{E}_c$. Let $Z=
  L(\hat{a})$, where $\hat{a}$ is selected according to (A$_6$).
  Then, $(Z, \|\cdot\|_{\hat{a}})$ is an $AM$-space with $\hat{a}$
  as an order unit. Note that  $\hat{a} \in \|\cdot\|_{\hat{a}}$-int$Z_+$,
  $Z_+$ is $\|\cdot\|_{\hat{a}}$-closed in $Z$, and the
  $\|\cdot\|_{\hat{a}}$-closed unit ball of $Z$ coincides with
  the order interval $[-\hat{a}, \hat{a}]$.  Define a new
  economy $\hat{\mathcal{E}}$ which is identical with $\mathcal{E}$
  except for the commodity space being $Z$ equipped with the
  $\|\cdot\|_{\hat{a}}$-topology, each agent's consumption set
  being $Z_+$ in each state of nature $\omega \in \Omega$, and
  agent $i$'s ex ante expected utility being $V_i|_{Z^\Omega}$.
  If $(y_1,...,y_n)$ is a feasible allocation of $\mathcal{E}$,
  then $y_i(\omega)\in Z_+$ for each $i\in N$ and $\omega\in
  \Omega$. Thus, $x_i(\omega) \in Z_+$ for each $i\in N$ and
  $\omega\in \Omega$. Since $\sum_{i\in N}a_i(\omega)$ is an
  order unit of $Z$, $\sum_{i\in N}a_i(\omega)\in \|\cdot
  \|_{\hat{a}}$-int$Z_+$ for each $\omega\in \Omega$. Since
  $(Z, \|\cdot\|_{\hat{a}})$ is a Banach lattice, the
  $\|\cdot\|$-topology is weaker than the
  $\|\cdot\|_{\hat{a}}$-topology on $Z$. It follows that
  $U_i(\omega, \cdot)|_Z$ is $\|\cdot\|_{\hat{a}}$-continuous.
  Thus, we have verified that $\hat{\mathcal{E}}$ satisfies
  (A$_1$), (A$_2$) and (A$_4$), and $f$ is in the private core of
  $\hat{\mathcal{E}}_c$. By Theorem \ref{thm:core-wal1}, there
  is a non-zero positive element $\hat{\pi}\in ((Z, \|\cdot
  \|_{\hat{a}})^\ast)^\Omega$ such that $(f, \hat{\pi})$ is a
  non-trivial Walrasian expectations quasi-equilibrium in
  $\hat{\mathcal{E}}_c$. We need to show that there is a non-zero
  positive $\pi\in ((Z,\|\cdot\|)^\ast)^\Omega$ such that $(f,
  \pi)$ is a non-trivial Walrasian expectations quasi-equilibrium
  in $\mathcal{E}_c|_Z$, where $\mathcal{E}_c|_Z$ is identical with
  $\hat{\mathcal{E}}_c$ except for the commodity space being $Z$
  with the norm $\|\cdot\|$.

  Since $(f,\hat{\pi})$ is a non-trivial Walrasian expectations
  quasi-equilibrium in $\hat{\mathcal{E}}_c$, by Proposition
  \ref{prop:Equivalency}, $(x_1, ..., x_n, \hat{\pi})$ is a
  non-trivial Walrasian expectations quasi-equilibrium in
  $\hat{\mathcal{E}}$. Thus, $(x_1,...,x_n)$ is a privately
  Pareto optimal allocation in $\hat{\mathcal{E}}$, and also in
  $\mathcal{E}$. Let $i\in N$. By (A$_8$), there is a convex and
  $\|\cdot\|^\Omega$-open subset $W_i$ of $Y^\Omega$ such that
  $\emptyset\neq W_i\cap L_i \subseteq P_i(x_i)$ and
  $\|\cdot\|^\Omega$-${\rm cl} P_i(x_i)\subseteq
  \|\cdot\|^\Omega$-${\rm cl} W_i$. Since $\sum_{i\in N} a_i(\omega)$
  is a quasi-interior point of $Y_+$ for each $\omega\in \Omega$,
  $Z$ is $\|\cdot\|$-dense in $Y$. By Lemma \ref{lem:dense},
  $Z_+$ is $\|\cdot\|$-dense in $Y_+$. By definition,
  $L_i\cap Z_+^\Omega$ is
  $\|\cdot\|^\Omega$-dense in $L_i$. Thus $W_i\cap L_i\cap Z_+^
  \Omega\neq \emptyset$. Let $Q_i= W_i\cap Z^\Omega$ and $\hat{L}_i
  = L_i\cap Z^\Omega$. Then, $Q_i$ is convex and relatively
  $\|\cdot\|^\Omega$-open in $Z^\Omega$. Further, $\emptyset\neq
  Q_i\cap \hat{L}_i\subseteq \hat{P}_i(x_i)$ and
  $\|\cdot\|_Z^\Omega$-${\rm cl}\hat{P}_i(x_i)\subseteq \|\cdot
  \|_Z^\Omega$-${\rm cl}Q_i$, where $\hat{P}_i(x_i)= P_i(x_i)\cap
  Z^\Omega$. By (A$_2$), $x_i\in \|\cdot\|_Z^\Omega$-${\rm cl}
  \hat{P}_i(x_i)$, and so $x_i\in \|\cdot\|_Z^\Omega$-${\rm cl}Q_i$.
  For any $y_i\in Q_i\cap \hat{L}_i$,
  since $V_i(y_i)> V_i(x_i)$ and $(x_1, ..., x_n, \hat{\pi})$ is
  non-trivial Walrasian expectations quasi-equilibrium,
  $\sum_{\omega\in \Omega}\langle\hat{\pi}(\omega),
  y_i(\omega)\rangle\geq \sum_{\omega \in \Omega}\langle
  \hat{\pi}(\omega), x_i(\omega)\rangle$. Since $\Omega$ is finite,
  $\hat{\pi}\in ((Z,\|\cdot\|_{\hat{a}})^\Omega)^\ast$. By Lemma
  \ref{lem:contfunc}, there exist a $\pi_1^i\in ((Z,
  \|\cdot\|)^\Omega)^\ast$ and a linear functional $\pi_2^i$ on
  $(Z, \|\cdot\|)^\Omega$ such that $\sum_{\omega\in \Omega}
  \langle\pi_1^i(\omega), x_i(\omega)\rangle \leq \sum_{\omega
  \in \Omega}\langle\pi_1^i(\omega), y_i(\omega)\rangle$ for all
  $y_i\in Q_i$, $\sum_{\omega\in \Omega}\langle \pi_2^i(\omega),
  x_i(\omega)\rangle\leq \sum_{\omega\in \Omega} \langle\pi_2^i
  (\omega), y_i(\omega)\rangle$ for all $y_i\in \hat{L}_i$, and
  $\hat{\pi}= \pi_1^i+ \pi_2^i$. Since $\hat{L}_i$ is a cone,
  $\sum_{\omega\in \Omega}\langle\pi_2^i (\omega), x_i(\omega)
  \rangle= 0$. It follows that $\sum_{\omega
  \in \Omega} \langle\pi_2^i(\omega), y_i(\omega) \rangle\geq 0$
  for all $y_i\in \hat{L}_i$. Hence, we have
  \begin{enumerate}
  \item[(3.1)] $\sum_{\omega\in \Omega}
  \langle\hat{\pi}(\omega), x_i(\omega)\rangle= \sum_{\omega\in
  \Omega}\langle\pi_1^i(\omega), x_i(\omega) \rangle$,
  \item[(3.2)] $\sum_{\omega\in
  \Omega}\langle\hat{\pi}(\omega), y_i(\omega)\rangle \geq
  \sum_{\omega \in \Omega}\langle\pi_1^i(\omega), y_i(\omega)
  \rangle$ for all $y_i\in \hat{L}_i$.
  \end{enumerate}
  Since $(Z, \|\cdot\|_{\hat{a}})$ is a Banach lattice, $(Z,
  \|\cdot\|_{\hat{a}})^\ast$ agrees with the order dual of $Z$.
  In what follows,  let $(Z, \|\cdot\|_{\hat{a}}) ^\ast$ be
  endowed with the dual ordering relative to the ordering of
  $Z$. Since each $y_i\in \hat{L}_i$ can be written as $y_i=
  \sum_{S\in \mathcal{F}_i}y_i^S{\bf 1}_S$, where $y_i^S\in
  Z_+$ and $S\in \mathcal{F}_i$, from (3.1) and (3.2), it can be
  verified that the following hold for all $S\in \mathcal{F}_i$:
  \begin{enumerate}
  \item[(3.3)] $\sum_{\omega\in S}\pi_1^i(\omega)\leq \sum_{\omega
  \in S}\hat{\pi} (\omega)$,
  \item[(3.4)]
  $\sum_{\omega\in S}\langle\hat{\pi}(\omega), x_i(\omega)\rangle
  =\sum_{\omega\in S}\langle\pi_1^i(\omega), x_i(\omega)\rangle$.
  \end{enumerate}

  Since $(Z,\|\cdot\|)$ is a locally solid Riesz space, $(Z,
  \|\cdot\|)^\ast$ is an ideal in $(Z, \|\cdot\|_{\hat{a}})^\ast$.
  Pick an arbitrary element $S\in \mathcal{F}_i$. Since $\hat{\pi}
  (\omega)\geq 0$ for all $\omega\in S$, by (3.3) and
  Lemma \ref{lem:rieszdec}, there is an element $\tilde{\pi}^i
  \in ((Z, \|\cdot\|)^\ast)^S$ such that $\tilde{\pi}^i\leq
  \hat{\pi}$ on $S$ and $\sum_{\omega\in S}\tilde{\pi}^i
  (\omega)= \sum_{\omega\in S}\pi_1^i(\omega)$. We claim that for
  each $\omega\in S$, $\langle\tilde{\pi}^i(\omega), x_i(\omega)
  \rangle= \langle\hat{\pi}(\omega), x_i(\omega)\rangle$. To show
  this, let $x_i= \sum_{R\in \mathcal{F}_i}x_i^R {\bf 1}_R$, where
  $x_i^R\in Z_+$. By (3.4),
  \[
  \sum_{\omega\in S}\langle\tilde{\pi}^i(\omega), x_i(\omega)
  \rangle = \left\langle\sum_{\omega\in S}\pi_1^i(\omega),
  x_i^S\right\rangle = \sum_{\omega\in S}\langle\hat{\pi}
  (\omega), x_i(\omega)\rangle.
  \]
  Moreover, $\langle\tilde{\pi}^i(\omega), x_i(\omega)\rangle
  \leq\langle\hat{\pi}(\omega), x_i(\omega)\rangle$ for each
  $\omega\in S$. So, we must have $\langle\tilde{\pi}^i(\omega),
  x_i(\omega) \rangle= \langle\hat{\pi}(\omega), x_i(\omega)
  \rangle$ for each $\omega\in S$, and the claim is verified.
  Since $\mathcal{F}_i$ is a partition of $\Omega$, there is
  an element $\tilde{\pi}^i\in
  ((Z, \|\cdot\|)^\ast)^\Omega$ such that $\tilde{\pi}^i\leq
  \hat{\pi}$ on $\Omega$ and $\langle \tilde{\pi}^i
  (\omega), x_i(\omega)\rangle=\langle\hat{\pi}(\omega), x_i
  (\omega) \rangle$ for each $\omega \in \Omega$. Let $N_0=
  N\cup \{0\}$, $\tilde{\pi}^0(\omega)= 0$ and $a_0(\omega)= 0$
  for each $\omega\in \Omega$. Since $(Z, \|\cdot\|)^\ast$ is
  an ideal, for each $\omega \in \Omega$,
  we can choose an element $\ddot{\pi}(\omega)\in
  (Z, \|\cdot\|)^\ast$ such that $\ddot{\pi} (\omega)=
  \sup\{\tilde{\pi}^i(\omega): i\in N_0\}$. Then $\ddot{\pi}
  \in ((Z, \|\cdot\|)^\ast)^\Omega$, and $\ddot{\pi}\le
  \hat{\pi}$. Define $x_0\in Z_+^\Omega$ such that $x_0(\omega)=
  \sum_{i\in N}a_i(\omega)-\sum_{i\in N}x_i(\omega)$ for each
  $\omega\in \Omega$. By the Riesz-Kantorovich formulas, we
  obtain
  \begin{eqnarray*}
  \left\langle\ddot{\pi}(\omega), \sum_{i\in N}a_i(\omega)\right
  \rangle
  &=& \sup\left\{\sum_{i\in N_0}\langle\tilde{\pi}^i(\omega),
  y_i\rangle: y_i\in Z_+, \sum_{i\in N_0}y_i= \sum_{i\in N}
  a_i(\omega)\right\} \nonumber\\
  &\geq& \sum_{i\in N_0} \langle\tilde{\pi}^i(\omega), x_i(\omega)
  \rangle = \sum_{i\in N} \langle\tilde{\pi}^i(\omega),
  x_i(\omega)\rangle\nonumber\\
  &=& \left\langle\hat{\pi}(\omega), \sum_{i\in N}x_i(\omega)
  \right\rangle = \left\langle\hat{\pi}(\omega), \sum_{i\in N}
  a_i(\omega)\right\rangle.
  \end{eqnarray*}
  Applying $\ddot{\pi}\leq \hat{\pi}$, we have $\langle
  \ddot{\pi}(\omega), \sum_{i\in N}a_i(\omega)\rangle= \langle
  \hat{\pi}(\omega),
  \sum_{i\in N}a_i(\omega)\rangle$ for each $\omega\in \Omega$.
  Note that $Z= L(\sum_{i\in N}a_i(\omega))$ for each $\omega\in
  \Omega$. Let $z\in Z_+$ be fixed. Choose $\delta> 0$ be such
  that $z\leq \delta \sum_{i\in N}a_i(\omega)$ for each $\omega
  \in \Omega$. Then, $\langle \ddot{\pi}(\omega), (\delta
  \sum_{i\in N}a_i (\omega)- z)\rangle\leq \langle \hat{\pi}
  (\omega), (\delta\sum_{i\in N}a_i(\omega)- z)\rangle$, and so
  $\langle \ddot{\pi}(\omega), z\rangle\geq \langle \hat{\pi}
  (\omega), z\rangle$ for each $\omega\in \Omega$. Consequently,
  $\ddot\pi\geq \hat\pi$ and therefore, $\ddot\pi= \hat\pi$.
  Thus, $\hat{\pi}\in ((Z, \|\cdot\|)^\ast)^\Omega$ and
  $(f, \hat{\pi})$ is a non-trivial Walrasian expectations
  quasi-equilibrium in $\mathcal{E}_c|_Z$. By the Hahn-Banach
  theorem, we can choose a positive element $\pi\in ((Y, \|\cdot
  \|)^\ast)^\Omega$ such that $\pi$ is an extension of $\hat{\pi}$.
  Since $L_t\cap Z_+^\Omega$ is $\|\cdot\|^\Omega$-dense in $L_t$
  and $V_t$ is $\|\cdot\|^\Omega$-continuous for each $t\in I$,
  we deduce that $\sum_{\omega\in \Omega}\langle \pi(\omega),
  y(\omega)\rangle\geq \sum_{\omega\in \Omega}\langle\pi(\omega),
  a(t, \omega)\rangle$ for all $y\in L_t$ satisfying
  $V_t(y)> V_t(f(t, \cdot))$. Further, if $\sum_{\omega
  \in \Omega}\langle\pi(\omega), a(t, \omega)\rangle> 0$, then
  $\|\cdot\|^\Omega$-continuity of $V_t$ implies that $(f, \pi)$
  is a non-trivial Walrasian expectations quasi-equilibrium of
  $\mathcal{E}_c$. This completes the proof. 
  \end{proof}

  \begin{corollary} \label{coro:quasiWalras}
  Let $f$ be a feasible allocation such that $f(t, \cdot)= x_i$
  for all $t\in I_i$ and $i\in N$.
  Under {\rm (A$_1$)-(A$_4$)} and {\rm (A$_6$)-(A$_8$)}, $f$ is a
  Walrasian expectations allocation if and only if $f$ is in the
  private core of $\mathcal{E}_c$.
  \end{corollary}

  Now, we extend Theorem \ref{thm:core-wal2} to an asymmetric
  information economy with equal treatment property whose
  commodity space is a Banach lattice having no quasi-interior
  point in its positive cone. The following definition of strong
  ATY-properness and the argument to get continuity of the
  equilibrium price in the next theorem are similar to those
  in (A8$^\prime$) and Theorem 3 of \cite{Podczeck:08}.

  \begin{definition}
  The relation $P_i: L_i \rightrightarrows L_i$ is called
  \emph{strongly ATY-proper} at $x_i\in L_i$ if there is a convex
  subset $\widehat{P}_i (x_i)$ of $Y^\Omega$ with non-empty
  $\|\cdot\|^\Omega$-interior such that $\widehat{P}_i(x_i)\cap L_i=
  P_i(x_i)$ and ($\|\cdot\|^\Omega$-int$\widehat{P}_i(x_i)$)
  $\cap L_i\cap L(\sum_{i\in N} a_i)\neq \emptyset$.
  \end{definition}

  \begin{itemize}
  \item[(A$_9$)] \emph{Strong properness}. If $(x_1,...,x_n)$ is a
  privately Pareto optimal allocation in $\mathcal{E}$, then $P_i$
  is strongly ATY-proper at $x_i$ for each $i\in N$.
  \end{itemize}

  \begin{theorem}\label{thm:core-wal3}
  Assume that $\mathcal{E}$ satisfies \emph{(A$_1$)-(A$_3$)},
  \emph{(A$_4^\prime$)}, \emph{(A$_6$)} and \emph{(A$_9$)}. Let $f$ be
  a feasible allocation in $\mathcal{E}_c$ such that $f(t, \cdot)=
  x_i$ for all $t\in I_i$ and $i\in N$. If $f$ is in the private
  core of $\mathcal{E}_c$, then $(f, \pi)$ is a non-trivial Walrasian
  expectations quasi-equilibrium of $\mathcal{E}_c$ for some non-zero
  $\pi: \Omega \to Y^\ast_+$.
  \end{theorem}

  \begin{proof}
  Let $f$ be in the private core of $\mathcal{E}_c$. Let $Z=
  L(\hat{a})$, where $\hat{a}$ is selected according to (A$_6$).
  Then, $(X,\|\cdot\|)$ equipped with the ordering of
  $(Y, \|\cdot\|)$ is a Banach lattice, where $X$ denotes the
  $\|\cdot\|$-closure of $Z$ in $Y$. Note that for any feasible
  allocation $(y_1,...,y_n)$ of $\mathcal E$, $y_i(\omega)$ belongs
  to $Z_+$ for each $i\in N$ and $\omega\in \Omega$. In particular,
  $x_i(\omega)\in Z_+$ for each $i\in N$ and $\omega\in \Omega$.
  Clearly, for each $i\in N$ and $\omega \in \Omega$, $U_i(\omega,
  \cdot)|_X$ satisfies (A$_1$)-(A$_3$). Suppose that $(y_1,...,y_n)$
  is a privately Pareto optimal allocation in the economy
  $\mathcal{E}|_X$, which is identical with $\mathcal{E}$ except
  for the commodity space being $X$, each agent's consumption set
  being $X_+$ in each state of
  nature $\omega\in \Omega$, and agent $i$'s ex ante expected
  utility being $V_i|_{X^\Omega}$. Then $(y_1,...,y_n)$ is privately
  Pareto optimal in $\mathcal{E}$. Take $\widetilde{P}_i
  (x_i)= \widehat{P}_i(x_i)\cap X^\Omega$ for each $i\in N$, where
  $\widehat{P}_i(x_i)$ is chosen according to (A$_9$). So, for each
  $i\in N$, $\widetilde{P}_i(x_i)$ is convex with non-empty relative
  $\|\cdot\|^\Omega$-interior in $X^\Omega$. Let $\hat{L}_i= L_i
  \cap X^\Omega$ for each $i\in N$. By (A$_9$), for each $i\in N$,
  $\widetilde{P}_i(x_i)\cap \hat{L}_i= P_i(x_i)|_{X^\Omega}$, where
  $P_i(x_i)|_{X^\Omega}$ is the restriction of $P_i(x_i)$ to $X^\Omega$
  and $($int$\widetilde{P}_i(x_i))\cap \hat{L}_i\neq \emptyset$. Thus,
  (A$_1$)-(A$_4$), (A$_6$) and (A$_8$) are satisfied for the economy
  $\mathcal{E}|_X$. Note that $f$ is in the private core of
  $\mathcal{E}_c|_X$. By Theorem \ref{thm:core-wal2}, there exists a
  non-zero positive element $\pi\in (X^\ast)^\Omega$ such that $(f,
  \pi)$ is a non-trivial Walrasian expectations quasi-equilibrium in
  $\mathcal{E}_c|_ X$. Therefore, by Proposition \ref{prop:Equivalency},
  $(x_1,...,x_n, \pi)$ is a non-trivial Walrasian expectations
  quasi-equilibrium in $\mathcal{E}|_X$. By the Hahn-Banach theorem,
  there is a non-zero positive element $\hat{\pi}\in (Y^\ast)^\Omega$
  which is an extension of $\pi$. Then $(x_1,...,x_n, \hat{\pi})$
  satisfies all conditions of non-trivial Walrasian expectations
  quasi-equilibrium of $\mathcal E$ except for the fact that if
  $\sum_{\omega\in \Omega}\langle\hat{\pi}(\omega), a_i(\omega)
  \rangle\neq 0$, then $\sum_{\omega\in \Omega}\langle\hat{\pi}(\omega),
  y_i(\omega)\rangle> \sum_{\omega\in \Omega}\langle\hat{\pi}(\omega),
  a_i(\omega)\rangle$ for all $y_i\in L_i$ satisfying $V_i(y_i)>
  V_i(x_i)$. Clearly, this is true for all $y_i\in \hat{L}_i$.

  Since $(x_1,...,x_n, \pi)$ is a non-trivial Walrasian expectations
  quasi-equilibrium in $\mathcal{E}|_X$, $(x_1,...,x_n)$ is privately
  Pareto optimal in $\mathcal{E}|_X$ and hence, in $\mathcal{E}$. Pick
  an $i\in N$. By (A$_9$), there is a convex and
  $\|\cdot\|^\Omega$-open subset $Q_i$ of $Y^\Omega$ such that
  $\emptyset\neq Q_i\cap \hat{L}_i\subseteq P_i(x_i)|_{X^\Omega}$
  and $\|\cdot\|^\Omega$-cl$P_i(x_i)\subseteq$
  $\|\cdot\|^\Omega$-cl$Q_i$. By (A$_2$), $x_i\in
  \|\cdot\|^\Omega$-cl$P_i(x_i)$ and hence, $x_i\in
  \|\cdot\|^\Omega$-cl$Q_i$.
  For any $y_i\in Q_i\cap \hat{L}_i$, since $V_i(y_i)> V_i(x_i)$
  and $(x_1, ..., x_n, \hat{\pi})$ is a non-trivial Walrasian
  expectations quasi-equilibrium in $\mathcal{E}|_X$,
  $\sum_{\omega\in \Omega}\langle\hat{\pi}(\omega), y_i(\omega)\rangle
  \geq \sum_{\omega\in \Omega}\langle\hat{\pi}(\omega), x_i(\omega)
  \rangle$. Note that $\hat{L}_i$ is convex, $x_i\in \hat{L}_i$ and
  $\hat{\pi}\in (Y^\Omega)^\ast$. By an argument similar to that in
  Theorem \ref{thm:core-wal2}, we can find an element $\pi_1^i\in
  (Y^\Omega)^\ast$ such that $\sum_{\omega\in \Omega}\langle\pi_1^i
  (\omega), x_i(\omega) \rangle= \sum_{\omega\in \Omega}\langle\hat{\pi}
  (\omega), x_i(\omega)\rangle$, $\sum_{\omega\in \Omega}\langle\pi_1^i
  (\omega), x_i(\omega)\rangle\leq \sum_{\omega\in \Omega}\langle
  \pi_1^i(\omega), y_i(\omega)\rangle$ for all $y_i\in Q_i$, and
  $\sum_{\omega\in \Omega}\langle \pi_1^i(\omega), y_i(\omega)\rangle
  \leq \sum_{\omega\in \Omega} \langle\hat{\pi}(\omega), y_i(\omega)
  \rangle$ for all $y_i\in \hat{L}_i$. Since $\pi_1^i$ is
  $\|\cdot\|^\Omega$-continuous and $\|\cdot
  \|^\Omega$-cl$P_i(x_i)\subseteq \|\cdot\|^\Omega$-${\rm cl}Q_i$,
  $\sum_{\omega\in \Omega} \langle\pi_1^i(\omega), x_i(\omega)
  \rangle\leq \sum_{\omega\in \Omega}\langle\pi_1^i(\omega),
  y_i(\omega)\rangle$ for all $y_i\in P_i(x_i)$. Now, consider elements
  $\pi_\star^i, \pi^\star\in  (Y^\Omega)_+^\ast$ defined by $\langle
  \pi_\star^i, y_i\rangle= \sum_{\omega\in \Omega}\langle \pi_1^i
  (\omega), y_i^S(\omega) \rangle$ and $\langle \pi^\star, y_i\rangle=
  \sum_{\omega\in \Omega}\langle \hat{\pi}(\omega), y_i^S(\omega)
  \rangle$, where $y_i^S= \frac{1}{|S|}\sum_{\omega\in S}y_i(\omega)$
  for $S\in \mathcal{F}_i$. Then $\tilde{\pi}^i= \pi_\star^i+ \hat{\pi}-
  \pi^\star \in (Y^\Omega)^\ast$, and it can be verified that
  \begin{itemize}
  \item [(3.5)] $\sum_{\omega\in \Omega}\langle\tilde{\pi}^i(\omega),
  x_i(\omega)\rangle= \sum_{\omega\in \Omega}\langle\hat{\pi}
      (\omega), x_i(\omega)\rangle$,
  \item [(3.6)] $\sum_{\omega\in \Omega}\langle\tilde{\pi}^i (\omega),
  x_i(\omega)\rangle\leq \sum_{\omega\in \Omega}\langle\tilde{\pi}^i
       (\omega), y_i(\omega)\rangle$ for all $y_i\in P_i(x_i)$, and
  \item [(3.7)] $\sum_{\omega\in \Omega}\langle \tilde{\pi} ^i(\omega),
  z(\omega) \rangle\leq \sum_{\omega\in \Omega} \langle\hat{\pi}
  (\omega), z(\omega)\rangle$ for all $z\in X_+^\Omega$.
  \end{itemize}
  Since $Y$ is a locally solid Riesz space, $Y^\ast$ is an ideal
  in the order dual of $Y$. Let $N_0= N\cup \{0\}$, $\tilde{\pi}^0
  (\omega)= 0$ and $a_0(\omega)= 0$ for each $\omega\in \Omega$.
  Define $\ddot{\pi}\in (Y^\ast)^\Omega$ such that $\ddot{\pi}(\omega)
  =\sup\{\tilde{\pi}^i(\omega): i\in N_0\}$ for each $\omega \in \Omega$,
  and $x_0\in Z_+^\Omega$ such that  $x_0(\omega)= \sum_{i\in N}a_i
  (\omega)- \sum_{i\in N}x_i(\omega)$ for each $\omega\in \Omega$.
  By the Riesz-Kantorovich formulas and techniques similar to those in
  Theorem \ref{thm:core-wal2}, we have $\left\langle\ddot{\pi}(\omega),
  \sum_{i\in N}a_i(\omega)\right\rangle \geq \sum_{i\in N}\left\langle
  \tilde{\pi}^i (\omega), x_i(\omega)\right\rangle$ for each $\omega\in
  \Omega$. Using (3.5), we can obtain $\sum_{\omega\in \Omega}\left
  \langle\ddot{\pi}(\omega), \sum_{i\in N}a_i(\omega) \right\rangle
  \geq\sum_{\omega\in \Omega}\left\langle\hat{\pi}(\omega), \sum_{i\in N}
  a_i(\omega)\right\rangle.$ Moreover, the Riesz-Kantorovich formulas
  and (3.7) imply $\sum_{\omega\in \Omega}\langle\ddot{\pi}(\omega),
  z(\omega)\rangle\leq \sum_{\omega\in \Omega}\langle \hat{\pi}
  (\omega), z(\omega)\rangle$ for all $z\in X_+^\Omega$. Since
  $Z^\Omega= L(\sum_{i\in N}a_i)$, $\ddot{\pi} \equiv \hat{\pi}$ on
  $Z^\Omega$, which can be combined with (3.5) and (3.6)
  to derive $\sum_{\omega\in \Omega}\langle\ddot{\pi}(\omega), x_i
  (\omega)\rangle\leq \sum_{\omega\in \Omega}\langle\ddot{\pi}
  (\omega), y_i(\omega)\rangle$ for all $y_i\in P_i(x_i)$. It follows
  from (A$_2$) and the fact that $(x_1, ..., x_n, \ddot\pi)$ is a
  non-trivial Walrasian expectations equilibrium in ${\mathcal E}|_X$,
  $\sum_{\omega\in \Omega}
  \langle\ddot{\pi}(\omega), x_i(\omega)\rangle= \sum_{\omega\in
  \Omega}\langle\ddot{\pi}(\omega), a_i(\omega)\rangle$. In case
  that $\sum_{\omega \in \Omega}\langle\pi(\omega), a_i(\omega)
  \rangle\neq 0$, by (A$_1$),
  $\sum_{\omega\in \Omega}\langle\ddot{\pi}(\omega), a_i(\omega)
  \rangle < \sum_{\omega\in \Omega}\langle\ddot{\pi} (\omega), y_i
  (\omega)\rangle$ for all $y_i\in P_i(x_i)$. Thus, $(x_1,...,x_n,
  \ddot{\pi})$ is a non-trivial Walrasian expectations quasi-equilibrium
  in $\mathcal E$. Hence, by Proposition $\ref{prop:Equivalency}$, $(f,
  \ddot{\pi})$ is a non-trivial Walrasian expectations
  quasi-equilibrium in $\mathcal{E}_c$. 
  \end{proof}

  \begin{corollary}
  Let $f$ be a feasible allocation such that $f(t, \cdot)= x_i$ for
  all $t\in I_i$ and $i\in N$. Under {\rm (A$_1$)-(A$_3$)}, {\rm
  (A$_4^\prime$)}, {\rm (A$_6$)-(A$_7$)} and {\rm (A$_9$)}, $f$ is
  a Walrasian expectations allocation if and only if $f$ is in the
  private core of $\mathcal{E}_c$.
  \end{corollary}

  \subsection{\it Blocking non-private core allocations}
  \label{sec:private core}

  In this subsection, we give an extension of Vind's theorem to an
  asymmetric information economy with a continuum of agents
  having the equal treatment property and whose commodity space
  is a Banach lattice, using the following lemma instead of any
  convexity theorem on the commodity space.

  \begin{lemma} \label{lem:lemmaforvind}
  Assume that $\mathcal{E}$ satisfies \emph{(A$_1$)},
  \emph{(A$_3$)} and \emph{(A$_5$)}. Let $f$ be an allocation
  such that $f(t, \cdot)= f_i$ for all $t\in I_i$ and $i\in N$.
  If $f$ is privately blocked in $\mathcal{E}_c$, then it is
  privately blocked by a coalition $A$ via a function $g$ such
  that $g(t, \cdot)= g_i\in L_i$ if $t\in A\cap I_i$ and $i\in N$,
  and $\int_A(a(t,\omega)-g(t,\omega))d\mu(t)\ge z$ for all
  $\omega\in\Omega$, where $z> 0$.
  \end{lemma}

  \begin{proof}
  Since $f$ is privately blocked in ${\mathcal E}_c,$
  there are a coalition $\hat{A}\subseteq I$ and an $\hat{h}:
  \hat{A}\times \Omega\rightarrow Y_+$ such that $\hat{h}(t,
  \cdot)\in L_t$ and $V_{t}(\hat{h}(t, \cdot))> V_{t}(f(t,
  \cdot))$ for all $t\in \hat{A}$, and $\int_{\hat{A}} \hat{h}
  (t, \omega) d\mu(t)\le\int_{\hat{A}}a(t, \omega)
  d\mu(t)$ for all $\omega \in \Omega$. Let
  $\hat{A}_i = \hat{A} \cap I_i$ for each $i \in N$, $\hat{N}
  =\{i\in N: \mu(\hat{A}_i) \neq 0\}$, and $A=\bigcup_{i\in
  \hat{N}}\hat{A}_i$. For each $i\in \hat{N}$ and $\omega\in
  \Omega$, put $h_i(\omega) = \frac{1}{\mu(\hat{A}_i)}
  \int_{\hat{A}_i} \hat{h}(t,\omega) d\mu(t)$. Then $h_i\in L_i$
  for all $i\in \hat{N}$. Define $h: A\times \Omega\rightarrow
  Y_+$ by $h(t, \omega)=h_i(\omega)$ if $t\in \hat{A}_i$. Clearly,
  for each $\omega\in \Omega$, $\int_{A}h(t, \omega)d\mu(t)\leq
  \int_{A}a(t, \omega)d\mu(t)$, equivalently, $\sum_{i\in \hat{N}}
  h_i(\omega)\mu(\hat{A}_i)\leq\sum_{i\in \hat{N}}a_i
  (\omega)\mu(\hat{A}_i)$. Moreover, concavity of $V_i$ implies
  that $V_i(h(t, \cdot))> V_i(f(t, \cdot))$ for all $t\in
  \hat{A}_i$ and all $i\in \hat{N}$.

  Choose a sequence $\{c_m\}\subseteq (0, 1)$ with $c_m\to 0$.
  For each $i\in \hat{N}$ and each integer $m \ge 1$, define a
  function $g_i^m: \Omega \to Y_+$ by $g_i^m(\omega)=
  (1-c_m)h_i(\omega)$. Then $g_i^m \in L_i$ for all $i\in
  \hat{N}$. For each $i\in \hat{N}$, since $h_i\in P_i(f_i)$
  and $P_i(f_i)$ is $\|\cdot\|^\Omega$-open, then $g_i^m\in
  P_i(f_i)$ whenever $m$ is sufficiently large. If
  we choose such an $m$, then
  $\sum_{i\in \hat{N}}g_i^m(\omega)\mu(\hat{A}_i)
  \le (1-c_m)\sum_{i\in \hat{N}}a_i(\omega)\mu(\hat{A}_i)$,
  and
  \[
  \sum_{i\in \hat{N}}a_i(\omega)\mu(\hat{A}_i)-  \sum_{i\in
  \hat{N}} g_i^m(\omega)\mu(\hat{A}_i)
  \geq c_m\sum_{i\in \hat{N}}a_i(\omega)\mu(\hat{A}_i).
  \]
  Let $z= \inf\left\{c_m \sum_{i\in \widehat{N}}a_i(\omega)
  \mu(\hat{A}_i): \omega\in \Omega\right\}$. By (A$_5$), we
  have $z> 0$. If we define $g: A\times \Omega\to Y_+$ such
  that $g(t, \omega)= g_i^m(\omega)$ for all $t\in \hat{A}_i$,
  then $\int_A(a(t,\omega)-g(t,\omega))d\mu(t)\geq z$ for
  all $\omega\in\Omega$ and $g(t,\cdot)\in L_t$ for all $t
  \in A$, which is required by the lemma. \qed
  \end{proof}

  \begin{theorem}\label{thm:infintevind}
  Assume that $\mathcal{E}$ satisfies {\rm (A$_1$)}-{\rm (A$_3$)}
  and {\rm (A$_5$)}. Let $f$ be a feasible allocation in
  $\mathcal{E}_c$ such that $f(t,\cdot)= f_i$ for all $t\in I_i$
  and $i\in N$. If $f$ is not in the private core of
  $\mathcal{E}_c$, then for any $0<\epsilon<1$, there is a
  coalition $S$ with $\mu(S)=\epsilon$ privately blocking $f$.
  \end{theorem}

  \begin{proof}
  Since $f$ is not in the private core of ${\mathcal E}_c$,
  by Lemma \ref{lem:lemmaforvind}, there is a coalition $A
  \subseteq I$ that privately blocks $f$ via a function $g: A
  \times \Omega \to Y_+$ such that $g(t,\cdot)=g_i\in L_i$,
  if $t\in \hat{A}_i$ and $\int_{A}(a(t, \cdot)-g(t, \cdot))
  d\mu(t) \geq z$, where $\hat{A}_i = A \cap I_i$ for all $i
  \in N$. Choose a $\delta\in(0, 1)$. Since $\mu$ is atomless,
  there is $E_i\subseteq \hat{A}_i$ such that
  $\mu(E_i)= \delta\mu(\hat{A}_i)$. Moreover, for any $t \in
  \hat{A}_i$, $a(t,\cdot)-g(t, \cdot)=a_i- g_i$. Hence,
  \[
  \int_{E_i}(a(t, \cdot)-g(t, \cdot))d\mu(t)
  = (a_i- g_i)\delta\mu(\hat{A}_i)
  = \delta\int_{\hat{A}_i}(a(t, \cdot)-g(t,\cdot))
  d\mu(t).
  \]
  Take $E=\bigcup_{i\in \hat{N}}E_i$. Then,
  $\int_{E}(a(t, \cdot)-g(t, \cdot))d\mu(t)
  = \delta\int_{A}(a(t, \cdot)-g(t, \cdot))d\mu(t)$, and
  $\mu(E)=\delta \mu(A)$.
  Since $\delta\int_{A}(a(t, \cdot)-g(t, \cdot))d\mu(t)> 0$
  for any $0< \delta< 1,$ there is a coalition $E\subseteq A$
  with $\mu(E)= \delta \mu(A)$ privately blocking $f$ via $g$.
  This proves the theorem for $\epsilon\leq \mu(A).$

  If $\mu(A)= 1,$ the proof has been completed. Otherwise, $\mu(I
  \setminus A)>0$, and for any given $0<\epsilon< 1$, we define
  a function $g_{\epsilon}: A\times \Omega\to Y_+$ such that
  $g_{\epsilon}(t,\omega)=\epsilon g(t,\epsilon)+
  (1-\epsilon) f(t, \omega)$. Then, $g_{\epsilon}(t,\cdot)\in
  L_t$ for all $t\in A$. By concavity of $V_{t}$, we have $V_t
  (g_{\epsilon}(t, \cdot))> V_{t}(f(t, \cdot))$ for all
  $t\in A.$ Let $\hat{C}= I\setminus A$, $C_i= \hat{C}\cap
  I_i$, $\hat{N}_1=\{i\in N:\mu(C_i)> 0\}$ and $C=
  \bigcup_{i\in \hat{N}_1} C_i$. Then $\mu(\hat{C})=\mu(C)$.
  If $i\in \hat{N}_1$, since $\mu$ is atomless, there is
  $B_i\subseteq C_i$ such that $\mu(B_i)=(1-\epsilon) \mu(C_i)$.
  Moreover, $a(t,\cdot)- f(t, \cdot)=a_i- f_i$ for any $t\in
  C_i$. Hence
  \[
  \int_{B_i}(a(t, \cdot)-f(t, \cdot))d\mu(t)
  = (1- \epsilon)\int_{C_i}(a(t, \cdot)-f(t, \cdot))d\mu(t).
  \]
  Let $B=\bigcup_{i\in \hat{N}_1}B_i$. Then $\mu(B)=
  (1- \epsilon)\mu(I\setminus A)$ and
  \[
  \int_{B}(a(t, \cdot)-f(t, \cdot))d\mu(t)
  = (1- \epsilon)\int_{I\setminus A}(a(t, \cdot)-
  f(t, \cdot))d\mu(t).
  \]
  Define a consumption bundle $h_{\epsilon}:B\times\Omega\to
  Y_+$ by $h_{\epsilon}(t, \omega)= f(t, \omega)+
  \frac{\epsilon\mu(A)}{\mu(B)}z$. Since $f(t, \cdot)\in L_t$
  and $z$ is constant, $h_{\epsilon}(t, \cdot)\in L_t$ for
  all $t\in B.$ By monotonicity of $V_t$, $V_{t}(h_{\epsilon}
  (t, \cdot))>V_{t}(f(t, \cdot))$ for all $t\in B.$

  Let $S=A\cup B$. Then, $\mu(S)=\mu(A) +(1- \epsilon)
  \mu(I\setminus A).$ Now we show that $S$ privately blocks $f$.
  Define a consumption bundle $y_{\epsilon}: S\times \Omega
  \to Y_+$ such that $y_{\epsilon}(t, \omega) = g_{\epsilon}(t,
  \omega)$ if $(t, \omega)\in A\times \Omega$, and $y_{\epsilon}
  (t, \omega)= h_{\epsilon}(t,\omega)$, if $(t, \omega)\in B\times
  \Omega$. Clearly, $y_\epsilon(t, \cdot)\in L_t$
  and $V_t(y_\epsilon(t, \cdot))> V_{t}(f(t, \cdot))$ for all
  $t\in S.$ It remains to verify that $y_{\epsilon}$ is
  feasible for $S$. Since
  $\int_{A}(a(t, \cdot)-g(t, \cdot))d\mu(t)\geq z$, we have
  \begin{eqnarray*}
  \int_{S}(a(t, \cdot)-y_{\epsilon}(t, \cdot))d\mu(t)
  &\ge& \epsilon z+ (1-\epsilon) \int_{A}(a(t, \cdot)-f(t, \cdot))
  d\mu(t)\\
  &&+\int_{B}(a(t, \cdot)-f(t, \cdot))d\mu(t)- \epsilon\mu(A)z.
  \end{eqnarray*}
  On the other hand, $\epsilon z-\epsilon\mu(A)z =\epsilon
  \mu(I\setminus A)z>0$, and
  \[
  \int_{B}(a(t, \cdot)-f(t, \cdot))d\mu(t) =
  (1- \epsilon)\int_{I\setminus A}(a(t, \cdot)-f(t, \cdot))d\mu(t).
  \]
  Combining the previous few inequalities and equalities, we
  obtain
  \[
  \int_{S}(a(t, \cdot)-y_{\epsilon}(t, \cdot))d\mu(t)
  >(1- \epsilon)\int_I (a(t, \cdot)-f(t, \cdot))d\mu(t)\geq 0,
  \]
  which verifies that $S$ privately blocks $f$ via $y_{\epsilon}$.
  \end{proof}

  \section{Characterizations of Walrasian expectations equilibrium}
  \label{sec:Equitheodiseco}

  In this section, we provide characterizations of Walrasian
  expectations equilibrium in terms of the private blocking power
  of the grand coalition. The following concepts for asymmetric
  information economies in \cite{Herves-Beloso-Moreno-Garcia-Yannelis:05b} 
  will be needed.

  \begin{definition}
  A coalition $S\subseteq N$ \emph{privately blocks an allocation}
  $x$ of $\mathcal{E}$ in the sense of Aubin via $y=(y_i)_{i \in
  S}$ if for all $i\in S$,
  there is $\alpha_i\in (0, 1]$ such that $V_i(y_i)> V_i(x_i)$
  and $\sum_{i\in S} \alpha_i y_i\leq \sum_{i\in S}\alpha_i a_i$.
  The \emph{Aubin private core} of $\mathcal{E}$ is the set of all
  feasible allocations which cannot be blocked by any coalition in
  the sense of Aubin, and an allocation $x$ in $\mathcal{E}$
  is called an \emph{Aubin dominated allocation} if $x$ is
  privately blocked by the grand coalition in the sense of Aubin.
  \end{definition}

  The proof of our first theorem in this section is similar to
  that Theorem 4.2 in \cite{Herves-Beloso-Moreno-Garcia-Yannelis:05b}.

  \begin{theorem}\label{thm:fuzzycoreequilibria1}
  Assume that $\mathcal{E}$ satisfies {\rm (A$_1$)-(A$_2$)},
  {\rm (A$_4$)-(A$_5$)} and {\rm (A$_7$)}, and its commodity
  space has an interior point in its positive
  cone. Then, a feasible allocation $x$ in $\mathcal E$ is a
  Walrasian expectations allocation if and only if $x$ is Aubin 
  non-dominated.
  \end{theorem}

  \begin{remark}[\cite{Herves-Beloso-Moreno-Garcia-Yannelis:05b}]
  \label{rem:compart}
  The participation of an agent $i$ in the grand coalition of
  $\mathcal{E}$ is said to be \emph{close to complete} if
  $\tilde{\alpha}_i> 1- \delta$ for sufficiently small
  $\delta>0$. Indeed, Theorem \ref{thm:fuzzycoreequilibria1}
  shows that the participation of each agent can be chosen to
  be close to complete. To see this, for any given $0<\delta<
  1$, by Theorem \ref{thm:infintevind}, we can choose a
  privately blocking coalition $S$ such that $\mu(S)> 1-
  \frac{\delta}{n}$. Hence, $\tilde{\alpha}_i= n\mu(S_i)> 1-
  \delta$ for all $i\in N$.
  \qed
  \end{remark}

  \begin{remark}\label{rem:fuzzycoreequilibria2}
  Conclusions of Theorem \ref{thm:fuzzycoreequilibria1} and
  Remark \ref{rem:compart} are also true under (A$_1$)-(A$_8$),
  or A$_1$)-(A$_3$), (A$_5$)-(A$_7$) and (A$_9$)), respectively.
  \qed
  \end{remark}

  The proof of our next theorem is similar to that of Theorem
  $4.1$ in \cite{Herves-Beloso-Moreno-Garcia-Yannelis:05a,
  Herves-Beloso-Moreno-Garcia-Yannelis:05b}. For an
  allocation $x=(x_1, ...,x_n)$ in ${\mathcal E}$ and a vector
  $r= (r_1,...,r_n) \in [0,1]^n$, consider an asymmetric
  information economy $\mathcal{E}(r, x)$ which is identical
  with $\mathcal{E}$ except for the random initial endowment of
  each agent $i$ being $a_i(r_i, x_i)= r_i a_i+ (1-r_i) x_i$.

  \begin{theorem}\label{thm:radequnondom}
  Assume that $\mathcal{E}$ satisfies {\rm (A$_1$)-(A$_2$)},
  {\rm (A$_4$)-(A$_5$)} and {\rm (A$_7$)}, and its commodity
  space has an interior point in its positive cone. Then, a
  feasible allocation $x$ in $\mathcal E$ is a Walrasian
  expectations allocation if and only if $x$ is privately
  non-dominated in $\mathcal{E}(r,x)$ for any $r\in [0, 1]^n$.
  \end{theorem}

  \begin{remark}[\cite{Herves-Beloso-Moreno-Garcia-Yannelis:05b}]
  \label{rem:comblock}
  Note that for each agent $i\in N$, $\tilde{r}_i$ can be selected
  arbitrarily close to $1$. To see this, for any given
  $0<\delta< 1$, by Theorem \ref{thm:infintevind}, we can choose
  a privately blocking coalition $S$ such that $\mu(S)> 1- 
  \frac{\delta}{n}$. Hence, $\tilde{r}_i= n\mu(S_i)> 1- \delta$
  for all $i\in N$. \qed
  \end{remark}

  \begin{remark}\label{rem:radequnondom1}
  Conclusion of Theorem \ref{thm:radequnondom} is also true
  under (A$_1$)-(A$_8$), or A$_1$)-(A$_3$), (A$_5$)-(A$_7$) and
  (A$_9$)), respectively. The same is true
  for Remark \ref{rem:comblock}. \qed
  \end{remark}

  \begin{remark}[\cite{Herves-Beloso-Moreno-Garcia-Yannelis:05b}]
  Theorem \ref{thm:radequnondom} and the first part of the Remark
  \ref{rem:radequnondom1} not only provide characterizations of
  Walrasian expectations allocation in terms of veto power of
  the grand coalition but also implies two welfare theorems as
  particular cases. If $x$ is a Walrasian expectations allocation
  in $\mathcal{E}$, then $x$ is a privately Pareto optimal
  allocation in $\mathcal{E}$ and any $\mathcal{E}(r, x)$
  with $r\in [0, 1]^n$. Thus, the first welfare theorem
  is an immediate consequence of Theorem \ref{thm:radequnondom}
  and the first part of Remark \ref{rem:radequnondom1}.

  Observe that if $x$ is privately Pareto optimal in
  $\mathcal{E}$, then $x$ is also privately Pareto optimal
  in $\mathcal{E}(0, x)$. If we choose $x= a$, then all
  economies $\mathcal{E}(r, x)$ are equal to $\mathcal{E}(0, x)$
  and $x$ is not privately blocked by the grand coalition.
  If $\inf\{x_i(\omega): \omega\in \Omega\}> 0$ for all $i\in N$,
  Theorem \ref{thm:radequnondom} and the first part of Remark
  \ref{rem:radequnondom1} implies that $x$ is
  a Walrasian expectations allocation, which is exactly the
  second welfare theorem. \qed
  \end{remark}

  We conclude this section with some basic examples of
  Banach lattices and a chart containing a summary where these
  Banach lattices are applied: (i) $\mathbb {R}^n$: the 
  Euclidean space; (ii) $\ell_\infty$: the space of real 
  bounded sequence with the supremum norm; (iii) $L_\infty 
  (\Omega, \Sigma, \mu)$: the space of essentially bounded, 
  measurable functions on a measure space $(\Omega, \Sigma, 
  \mu)$ with the essential supremum norm; (iv) $C(K)$: the 
  space of real valued continuous functions on a compact 
  Hausdorff space $K$ with the supremum norm; (v) $\ell_p$: 
  the space  of real sequences $x= \{x_n\}$ with the 
  norm $\|x\|_p=\left(\sum_{n=1}^\infty |x_n|^p
  \right)^{\frac{1}{p}}<\infty$ , where $1\leq p< \infty$;
  (vi) $L_p (\Omega, \Sigma, \mu)$: the space of measurable
  functions $f$ on the measure space $(\Omega, \Sigma, \mu)$ 
  with the norm $\|f\|_p=\left(\int_\Omega |f(\omega)|^p 
  d\mu(\omega) \right) ^{\frac{1}{p}}<\infty$, where 
  $1\leq p< \infty$; (vii) $M(K)$: the space of regular 
  Borel measures on a compact Hausdorff space $K$ with the 
  total variation norm.
  
  \[
  \begin{tabular}{|c|c|c|c|}
  \hline
  \multirow{2}{*}{$Y$} & $Y_+$ having & $Y_+$ having &
  \multirow{2}{*}{Applications}\\
   &  interior points & quasi-interior points
  & \\ \hline
  \multirow{2}{*}{$\mathbb R^n$} & \multirow{2}{*}{+} &
  \multirow{2}{*}{+} & Theorems \\
  & & & 1, 4, 5, 6 \\ \hline
  \multirow{2}{*}{$\ell_\infty$} & \multirow{2}{*}{+} &
  \multirow{2}{*}{+} & Theorems \\
  & & & 1, 4, 5, 6 \\ \hline
  \multirow{2}{*}{$L_\infty(\Omega, \Sigma, \mu)$} & \multirow{2}{*}{+} &
  \multirow{2}{*}{+} & Theorems \\
  & & & 1, 4, 5, 6 \\ \hline
  \multirow{2}{*}{$C(K)$} & \multirow{2}{*}{+} &
  \multirow{2}{*}{+} & Theorems \\
  & & & 1, 4, 5, 6 \\ \hline
  $\ell_p$  &  \multirow{2}{*}{-} & \multirow{2}{*}{+} &
  Theorems 2, 4\\
  ($1\le p <\infty$) &  &  & Remarks 4, 6\\ \hline
  $L_p(\Omega, \Sigma, \mu)$ & \multirow{2}{*}{-} &
  \multirow{2}{*}{+} & Theorems 2, 4\\
  ($1\le p <\infty$) &  &  & Remarks 4, 6\\ \hline
  $M(K)$ & \multirow{2}{*}{-} &
  \multirow{2}{*}{-} & Theorems 3, 4\\
  ($K$ uncountable) &  &  & Remarks 4, 6\\ \hline
  \end{tabular}
  \]

  \vskip 2em
  \noindent {\bf Appendix }~\label{sec: proper}

  \vskip 2em
  \noindent
  (A.1) \emph{Results on economies with arbitrary probability
  spaces of states}

  \vskip 2em
  \noindent
  In this section, we apply results established in previous
  sections, particularly Theorem \ref{thm:infintevind}, to
  characterize the ex-post core and private core of an economy
  with an arbitrary probability space of states of nature.

  Let $\widetilde{\mathcal E}$ be a discrete economy identical
  with $\mathcal E$ except for the space of states of
  nature being a probability measure space $(\Omega, \Sigma, \nu)$,
  where $\Omega$ is an (infinite) arbitrary set of states of nature,
  the $\sigma$-algebra $\Sigma$ denotes the set of events and the
  probability measure $\nu$ denotes the \emph{common prior} of each
  agent $i\in N$. For any $x_{i}:\Omega\to Y_+$, the \emph{ex ante
  expected utility} of agent $i$ is given by $V_{i}(x_{i})= \int_\Omega
  U_{i} (\omega, x_{i}(\omega))\nu(d\omega)$. The concepts of an
  \emph{assignment} and an \emph{allocation} in $\widetilde{\mathcal
  E}$ are defined in the same way as that in $\mathcal E$. Further,
  an allocation (assignment) $x$ is said to be \emph{feasible} if
  $\sum_{i= 1}^{n} x_{i}(\omega) \leq \sum_{i= 1}^{n}a_{i}
  (\omega)$ for almost all $\omega\in \Omega$. A feasible
  assignment $x$ of $\mathcal{E}$ is called \emph{fine non-dominated}
  (resp. \emph{weak fine non-dominated}) if for all $i\in N$, $x_i$ is
  $\mathcal{F}_i$-measurable (resp. $\bigvee_{i\in N}
  \mathcal{F}_i$-measurable) and there is no feasible assignment
  $y= (y_1,...,y_n)$ such that $y_i$ is $\bigvee_{i \in N}
  \mathcal{F}_i$-measurable and $V_i(y_i)> V_i(x_i)$ for all $i\in N$.
  A \emph{coalition} is a non-empty subset $S$ of $N$. A coalition
  $S \subseteq N$ \emph{privately blocks an allocation}
  $x$ in $\widetilde{\mathcal E}$ if there exists $y=
  (y_i)_{i\in S}$ such that $y_i\in L_i$ and $V_i(y_i)> V_i(x_i)$
  for all $i\in S,$ and $\sum_{i\in S}y_i(\omega)\leq \sum_{i\in S}
  a_i(\omega)$ for almost all $\omega\in \Omega$. The \emph{private
  core} of the economy $\widetilde{\mathcal E}$ is the set of all
  feasible allocations which are not privately blocked by any
  coalition. A feasible allocation $x$ of $\widetilde{\mathcal E}$
  is in the \emph{ex-post core} \cite{Einy-Moreno-Shitovitz:01}
  of $\widetilde{\mathcal E}$ if there
  are no coalition $S \subseteq N$, $y= (y_i)_{i\in S}$, and
  $\Omega_0\in \Sigma$ such that $\nu(\Omega_0) > 0$, $\sum_{i\in S}
  y_i(\omega)\leq \sum_{i\in S} a_i(\omega)$ for almost all $\omega
  \in \Omega_0$, and $U_i(\omega, y_i(\omega))> U_i(\omega,
  x_i(\omega))$ for all $i\in S$ and for almost all $\omega\in
  \Omega_0$.

  Assume that $\widetilde{\mathcal{E}}_c$ is a continuum economy
  associated with $\widetilde{\mathcal E}$. Since $\mathcal{F}_t=
  \mathcal{F}_i$ for each agent $t\in I_i$, infinitely many states
  can be taken in $\widetilde{\mathcal E}_c$ even agents share
  their information. An \emph{assignment}
  in $\widetilde{\mathcal E}_c$ is a function $f:I\times \Omega \to
  Y_+$ such that $f(\cdot, \omega)\in L_1(\mu, Y_+)$ for almost all
  $\omega\in \Omega$. Let $L_t= L_i$ for all $t\in I_i$ and $i\in N$.
  An assignment $f$ in $\widetilde{\mathcal E}_c$ is called an
  \emph{allocation} if $f(t,\cdot)\in L_t$ for almost all $t\in I$. An
  allocation $f$ in $\widetilde{ \mathcal{E}}_c$ is {\it feasible} if
  $\int_I f(t, \omega)d\mu(t)\leq \int_I a(t, \omega)d\mu(t)$ for
  almost all $\omega\in \Omega$. A coalition $S$ \emph{privately
  blocks an allocation} $f$ in $\widetilde{\mathcal
  E}_c$ if there is $g: S\times \Omega\to Y_+$ such that $g(t,
  \cdot)\in L_t$ and $V_t(g(t,\cdot))> V_t(f(t,\cdot))$ for all
  $t\in S,$ and $\int_{S}g(t, \omega)d\mu(t)\leq
  \int_{S}a(t, \omega) d\mu(t)$ for almost all $\omega\in \Omega$.
  The \emph{private core} of $\widetilde{\mathcal E}_c$ is
  the set of feasible allocations not privately blocked
  by any coalition. A feasible allocation $f$ of $\widetilde{\mathcal
  E}_c$ is in the \emph{ex-post core} of $\widetilde{\mathcal E}_c$
  if there are no coalition $S \subseteq I$, function $g: S\times
  \Omega \to Y_+$, and $\Omega_0\in \Sigma$ satisfying $\nu(\Omega_0)
  > 0$, $\int_S g(t, \omega) d\mu(t)\leq \int_S a(t, \omega)d\mu(t)$
  for almost all $\omega\in \Omega_0$, $U_t(\omega, g(t,\omega))
  > U_t(\omega, f(t, \omega))$ for all $t\in S$ and for almost all
  $\omega \in \Omega_0$.

  Next, we establish a relation between the set of all fine
  non-dominated allocations and the ex-post core. This result is
  new even for finite dimensional commodity space. For an $\omega
  \in\Omega$, let $\widetilde{\mathcal E}_c(\omega)$ denote the
  full information economy whose commodity space being $Y_+$,
  the set of agents being $I$, the utility function and initial
  endowment of agent $t$ being $U_t(\omega, \cdot)$ and $a(t,
  \omega)$ respectively.

  \begin{theorem}\label{thm:Ex-postCore}
  Assume that $\widetilde{\mathcal E}$ satisfies \emph{(A$_1$)-(A$_3$)}
  and \emph{(A$_5$)}. Let $x$ be a feasible allocation of
  $\widetilde{\mathcal E}$. If $x$ is fine non-dominated in
  $\widetilde{\mathcal E}(r,x)$ for any $r\in [0, 1]^n$, then $x$
  is in the ex-post core of $\widetilde{\mathcal E}.$
  \end{theorem}

  \begin{proof}
  Suppose that $x$ is not in the ex-post core of $\widetilde{\mathcal
  E}$. Then, the simple function $f$ on $I$, defined by $f(t,\cdot)=
  x_i$ for all $t\in I_i$ and $i\in N$, is not in the ex-post core of
  $\widetilde{\mathcal E}_c$. Thus, there is $\Omega_0\in \Sigma$ with
  $\nu(\Omega_0)> 0$ such that $f(\cdot, \omega)$ is feasible and
  not in the core of $\widetilde{\mathcal E}_c(\omega)$ for almost
  all $\omega\in \Omega_0$. Let $A \in \bigvee_{i\in N}\mathcal{F}_i$
  with $\nu(\Omega_0\cap A)\neq 0$. Applying Theorem
  \ref{thm:infintevind} for each $\widetilde{\mathcal E}_c(\omega)$
  and the fact that $f(t, \cdot)$ and $a(t,\cdot)$ are constant on
  $\Omega_0 \cap A$ for all $t\in I$, there exist a coalition $S^A
  \subseteq I$ with $\mu(S^A)> 1-\frac{1}{n}$, and a function $g^A:
  S^A\times (\Omega_0\cap A)\to Y_+$ such that
  \begin{enumerate}
  \item[(A.1.1)] $g^A(t, \cdot) \in L_t$ for all $t\in S^A$ and
  $\int_{S^A} g^A (t,\omega)d\mu(t)\leq \int_{S^A} a(t,\omega)
  d\mu(t)$ for almost all $\omega\in \Omega_0\cap A$,
  \item[(A.1.2)] $U_{t}(\omega, g^A(t, \omega))> U_{t}(\omega, f(t,
  \omega))$ for all $t\in S^A$ and for almost all $\omega\in
  \Omega_0\cap A$.
  \end{enumerate}
  Pick an $\omega_0 \in \Omega_0 \cap A$ such that both (A.1.1) and
  (A.1.2) hold. Now, define a function $h^A: S^A\times A\to Y_+$ such
  that $h^A(t,\omega) =g^A(t, \omega_0)$ for all $(t, \omega) \in
  S^A \times A$. Then, $h^A(t, \cdot)$ is $\bigvee_{i= 1}^n
  \mathcal{F}_i$-measurable for all $t\in S^A$ and $h^A(\cdot,
  \omega)$ is Bochner integrable for all $\omega\in A$. Since $a(t,
  \cdot)\in L_i$ for all $t\in I_i$ and $i\in N$,
  $a(\cdot,\omega)= a(\cdot, \omega^\prime)$ for all
  $\omega, \omega^\prime \in A$. Hence, $\int_{S^A}h^A(t, \omega)
  d\mu(t)\leq \int_{S^A}a(t,\omega)d\mu(t)$ for all $\omega\in A.$

  Let $S_i^A= S^A\cap I_i$ and $r_i^A= n\mu(S_i^A)$ for all $i\in N$.
  Since $\mu(S^A)> 1-\frac{1}{n},$ we have $r_i^A> 0$ for all
  $i\in N$. Now for all $i\in N$, $y_i^A= \frac{1}{\mu(S_i^A)}
  \int_{S_i^A}h^A(t, \cdot) d\mu(t)$ is $\bigvee_{i= 1}^n
  \mathcal{F}_i$-measurable. By (A$_3$), $U_i(\omega, y_i^A(\omega))
  > U_i(\omega, x_i(\omega))$ for all $i\in N$ and $\omega\in A$.
  Further, following from $\sum_{i\in N}\mu(S_i^A)y_i^A\leq \sum_{
  i\in N}\mu(S_i^A)a_i$, we have $\sum_{i\in N} r_i^A y_i^A\leq
  \sum_{i\in N} r_i^A a_i$. If we put $z_i^A= r_i^A y_i^A+(1- r_i^A)
  x_i$ for all $i\in N$, then $z_i^A$ is $\bigvee_{i=1}^n
  \mathcal{F}_i$-measurable and
  \[
  \sum_{i\in N} z_i^A\leq \sum_{i\in N}\{r_i^A a_i+
  (1- r_i^A)x_i\} =\sum_{i\in N} a_i(r_i^A, x_i).
  \]
  By (A$_3$) again, $U_i(\omega, z_i^A(\omega))>U_i(\omega, x_i
  (\omega))$ for all $i\in N$ and $\omega\in A$. Now for each
  $i\in N$, the assignment $b_i:\Omega\to Y_+$ defined by $b_i
  (\omega) =  z_i^A$ if $\omega\in A$ and $b_i(\omega) =  x_i$ if
  $\omega\not \in A$, is $\bigvee_{i= 1}^n \mathcal{F}_i$-measurable,
  and $b= (b_1, ..., b_n)$ is feasible in $\widetilde{\mathcal E}
  (r^A, x)$, where $r^A= (r_1^A, ..., r_n^A)$. Since
  \[
  V_i(b_i)
  > \int_{A} U_i(\omega, x_i(\omega))\nu(d\omega) +
  \int_{\Omega\setminus A} U_i(\omega, x_i(\omega))\nu(d\omega)
  = V_i(x_i),
  \]
  $x$ is fine dominated by $b$ in $\widetilde{\mathcal E}(r^A, x)$.
  The proof is completed.
  \end{proof}

  We conclude this section with a characterization of the
  private core of $\widetilde{\mathcal E}$ by the set of all
  privately non-dominated allocations of $\widetilde{\mathcal E}
  (r,x)$, which is new even for finite dimensional commodity space.

  \begin{theorem}\label{thm:pricorenondom}
  Assume that $\widetilde{\mathcal E}$ satisfies
  \emph{(A$_1$)-(A$_3$)} and \emph{(A$_5$)}. Then $x$ is in
  the private core of $\widetilde{\mathcal E}$ if and only
  if $x$ is a privately non-dominated allocation in
  $\widetilde{\mathcal E}(r,x)$ for any $r\in [0, 1]^n$.
  \end{theorem}
  \begin{proof}
  Let $x$ be in the private core of $\widetilde{\mathcal E}$. Suppose
  that $x$ is privately dominated in $\widetilde{\mathcal E}(r,x)$
  for some $r= (r_1, ..., r_n)\in [0, 1]^n$. Then, there is an allocation
  $(y_1, ..., y_n)$ such that $V_i(y_i)> V_i(x_i)$ for each $i\in N$
  and $\sum_{i\in N}y_i(\omega)\le \sum_{i\in N}\{r_i a_i(\omega)+
  (1- r_i)x_i(\omega)\}$ for almost all $\omega\in \Omega$. Since $x$
  is feasible, $\sum_{i\in N}y_i(\omega)\le \sum_{i\in N} a_i(\omega)$
  for almost all $\omega\in \Omega$, a contradiction with the fact
  that $x$ is in the private core of $\widetilde{\mathcal E}$.

  Conversely, let $x$ be a privately non-dominated allocation in
  $\widetilde{\mathcal E}(r,x)$ for any $r\in [0, 1]^n$. Assume that
  $x$ is not in the private core of $\widetilde{\mathcal E}$. Then,
  the simple function $f$ on $I$ is not in the private core of
  $\widetilde{\mathcal E}_c$ Note
  that the conclusion of Theorem \ref{thm:infintevind} is true in
  $\widetilde{\mathcal E}_c$ under (A$_1$)-(A$_3$) and (A$_5$). Then,
  the conclusion can be derived by an argument similar to that in
  Theorem \ref{thm:radequnondom}. 
  \end{proof}

  Let $\widetilde{\mathcal E}(s)$ be a symmetric information economy
  which is identical with $\widetilde{\mathcal{E}}$  except for the
  information for each agent being $\mathcal{F}_i^s= \bigvee_{i\in N}
  \mathcal{F}_i$.

  \begin{corollary}\label{cor:weakfine}
  Assume that $\widetilde{\mathcal E}$ satisfies \emph{(A$_1$)-(A$_3$)}
  and \emph{(A$_5$)}. Then, a feasible allocation $x$ is weak fine
  dominated in $\mathcal{E}(r,x)$ for any $r\in [0, 1]^n$ if only
  if $x$ is in the private core of $\widetilde{\mathcal E}(s)$.
  \end{corollary}

  \begin{remark}
  In Section \ref{sec:Equitheodiseco}, we have applied the
  equivalence theorems and the Vind's type theorem in Section
  \ref{sec:privatecore}
  to establish characterizations of Walrasian expectations
  allocations for an asymmetric economy with finitely many states
  of nature. The perspective reader may wonder why such
  characterizations are not discussed in the scenario where
  infinitely many states of nature are considered. Although Vind's
  theorem is valid for an economy with an arbitrary probability
  space of states of nature, these equivalence theorem are only
  valid for the scenario where finitely many states of nature
  are considered. So, the techniques employed in Section
  \ref{sec:Equitheodiseco} do not allow us to establish similar
  results for characterizing Walrasian expectations allocations
  in the scenario with infinitely many states of nature. However,
  the difficulty may be overcome by applying new and different
  techniques, which could be done in the future work. \qed
  \end{remark}

  \vskip 2em
  \noindent
  (A.2) \emph{Discussion on ATY-properness}

  \vskip 2em
  \noindent
  In this subsection, we discuss about the ATY-properness of utility
  functions.
  Let $A_i(x_i)= \{y_i\in L_i: U_i(\omega, y_i(\omega))> U_i(\omega,
  x_i(\omega)) \mbox{ for each } \omega\in \Omega \}$ for all $x_i\in
  L_i$. So, $A_i(x_i)\subseteq P_i(x_i)$. By (A$_2$), $x_i\in$
  ${\rm cl} A_i(x_i)$. Then, it is clear from the proof of Theorem
  \ref{thm:core-wal2} that we used the following fact from (A$_8$):
  \begin{enumerate}
  \item[(A.2.1)] If $x$ is a privately Pareto optimal allocation,
  then there is a convex and $\|\cdot\|^\Omega$-open subset $W_i$ of
  $Y^\Omega$ such that $\emptyset \neq W_i\cap L_i \subseteq A_i(x_i)$
  and $\|\cdot\|^\Omega$-${\rm cl} A_i(x_i)\subseteq\|\cdot
  \|^\Omega$-${\rm cl} W_i$ for each $i\in N$.
  \end{enumerate}

  We now introduce the definition of ATY-properness of a random
  utility function similar to that of preference relation.

  \begin{definition}
  For $\omega\in \Omega$, the utility function $U_i(\omega, \cdot)$ is
  called \emph{ATY-proper} at $x\in Y_+$ if there exists a convex
  subset $\widetilde{A}_i(\omega, x)$ of $Y$ with non-empty
  $\|\cdot\|$-interior such that $\widetilde{A}_i(\omega, x)\cap
  Y_+= A_i(\omega, x)$ and (int$\widetilde{A}_i(\omega, x)$)$\cap
  Y_+\neq \emptyset$, where $A_i(\omega, x)= \{y\in Y_+: U_i(\omega, y)>
  U_i(\omega, x)\}$.
  \end{definition}

  We need the following assumptions of agents' utility functions.

  \begin{itemize}
  \item[(A$_8^\prime$)] \emph{State Properness}. If $(x_i)_{i\in N}$ is a
  privately Pareto optimal allocation in $\mathcal{E}$, then $U_i(\omega,
  \cdot)$ is ATY-proper at $x_i(\omega)$ for each $\omega\in \Omega$ and
  $i\in N$.
  \end{itemize}

  \begin{itemize}
  \item[(A$_8^{\prime\prime}$)] \emph{Measurability}. For each $i\in N$
  and $x\in Y_+$, $U_i(\cdot, x)$ is $\mathcal{F}_i$-measurable.
  \end{itemize}

  Replacing (A$_8$) with (A$_8^\prime$) and (A$_8^{\prime\prime}$),
  we can derive (A.2.1) by taking $W_i= \prod_{\omega\in \Omega}$
  int$\widetilde{A}_i (\omega, x_i(\omega))$ for each $i\in N$,
  where $\widetilde{A}_i (\omega, x_i(\omega))= \widetilde{A}_i
  (\omega^\prime, x_i(\omega^\prime))$ if $\omega, \omega' \in R
  \in \mathcal{F}_i$. Alternatively, similar to the $v$-properness
  of the preference relation in \cite{Aliprantis-Tourky:00}, one can
  assume in the definition of ATY-properness of $U_i(\omega, \cdot)$
  that $x_i(\omega)+y_i(\omega)\in$
  int$\widetilde{A}_i(\omega, x_i(\omega))$ for some $y_i\in L_i$ in the
  place of (int$\widetilde{A}_i(\omega, x)$)$\cap Y_+\neq \emptyset$
  and we obtain (A.2.1) by using (A$_8^\prime$) only. Analogous assumptions
  are not sufficient for (A$_9$), because the preference $P_i(x_i)$
  coming from expected utility plays an important role. However,
  (A$_8$) and (A$_9$) directly follow from the $v$-properness of the
  preference relation in a full information economy. To see this,
  consider a correspondence $S_i:Y_+^\Omega\rightrightarrows Y_+^\Omega$
  defined by $S_i(x)= \{y\in Y_+^\Omega: V_i(y)> V_i(x)\}$ for $i\in N$.
  By (A$_3$), $S_i(x_i)$ is convex. Choose a non-zero element $v_i\in L_i$.
  By the $v_i$-properness assumption, there is a convex-valued correspondence
  $\widetilde{S}_i:Y_+^\Omega \rightrightarrows Y^\Omega$ such that for
  each $x\in Y_+^\Omega$ the following hold: (i) $x+ v_i\in
  \|\cdot\|^\Omega$-int$\widetilde{S}_i(x)$; and (ii) $\widetilde{S}_i(x)\cap
  Y_+^\Omega= S_i(x)$. Note that if $x\in L_i$, then we have $x+ v_i\in
  (\|\cdot\|^\Omega$-int$\widetilde{S}_i(x))\cap L_i$ and $\widetilde{S}_i(x)\cap
  L_i= P_i(x)$, and so $P_i$ is ATY-proper. Applying a similar argument,
  one can derive the strong ATY-proper preference relation by taking
  $v_i \in L_i \cap L(\sum_{i\in N} a_i)$ for all $i\in N$.

  \vskip 2em
  \noindent
  (A.3) \emph{Mathematical preliminaries}
  \vskip 2em
  \noindent
  Given a real vector space $X$, a function $\|\cdot\|:X \to [0,\infty)$
  satisfying (i) $\|x\|\geq 0$ for all $x\in X$ and $\|x\|= 0$ if
  and only if $x= 0$; (ii) $\|\alpha x\|= |\alpha|\|x\|$ for all
  $x\in X$ and $\alpha\in\mathbb{R}$; (iii)  $\|x+ y\|\leq \|x\|+
  \|y\|$ for all $x, y\in X$, is called a \emph{norm} on $X$, and
  $(X, \|\cdot\|)$ is called a \emph{normed space}. If $A$ and $Z$ are
  subsets of a normed space $(X, \|\cdot\|)$ with $A \subseteq Z$,
  the closure and the interior of $A$ in the relative topology
  generated by norm on $Z$ are denoted by $\|\cdot \|_Z$-cl$A$ and
  $\|\cdot \|_Z$-int$A$ respectively. When $Z=X$, without confusion,
  we simply write $\|\cdot \|$-cl$A$ and $\|\cdot \|$-int$A$ instead.
  In addition, if $A$ is convex with $\|\cdot \|$-int$A \ne
  \emptyset$, then $\|\cdot \|$-cl$A$ =
  $\|\cdot \|$-cl($\|\cdot \|$-int$A$). If $X$ is a real
  vector space and $\leq$ is a partial order on $X$, then the pair
  $(X, \leq)$ is called an \emph{ordered vector space} if for any
  $x, y, z\in X$ and any positive real number $\alpha$, $x\geq y$
  implies that $\alpha x+ z\geq \alpha y+ z$. Recall that a \emph{Riesz
  space} is an ordered vector space that is also a lattice, that is,
  every pair of elements $x, y\in X$ has a supremum $x \vee y$ and
  an infimum $x\wedge y$. For any element $x$ of a Riesz space,
  $|x|$ stands for the absolute value of $x$ and is expressed in
  the form $|x|= x^++x^-$, where $x^+= x\vee 0$ and $x^-= (-x)\vee 0$
  are positive and negative parts of $x$ respectively.
  Note that $x= x^+-x^-$ and $x^+\wedge x^-= 0$. An element $x\in X$
  is called a \emph{positive element} of $X$ if $x\geq 0$ and
  $X_+=\{x\in X: x \geq 0\}$. We write $x > 0$, if $x\in X_+$ and
  $x\ne 0$. For two points $x,y \in Y$, $x>y$ means $x-y >0$. If
  $x\leq y$ in $X$, then $[x, y]$ denotes the order interval
  $\{z\in X: x\leq z\leq y\}.$ A subset $A$ of $X$ is
  called \emph{order bounded} if $A\subseteq [x, y]$ for some $x, y
  \in X$. A norm is called a \emph{lattice norm} if $|x|< |y|$
  implies $\|x\|\leq \|y\|$. A \emph{normed Riesz space} is a Riesz
  space with a lattice norm. A complete normed Riesz space
  is called a \emph{Banach lattice}. A lattice norm on a Riesz
  space is an \emph{$M$-norm} if $\|x \vee y\|= \max\{\|x\|, \|y\|\}$
  for any $x, y\geq 0$. An \emph{$M$-space} is a normed Riesz space
  with an $M$-norm. A norm complete $M$-space is an \emph{$AM$-space}.
  A subset $A$ of a Riesz space is called \emph{solid} if $|x|< |y|$
  and $y\in A$ imply $x\in A$. A solid vector subspace of a Riesz
  space is called an \emph{ideal}. A pair $(X, \tau)$ is called a
  \emph{locally solid Riesz space} if $X$ is a Riesz space and $\tau$
  is a linear topology on $X$ such that $\tau$ has a basis at zero
  consisting of solid neighborhoods.

  A linear functional $f: X\to \mathbb R$ on a Riesz space $X$
  is called \emph{ordered bounded} if $f(A)$ is bounded in $\mathbb
  R$ for any order bounded subset $A$ of $X$. Recall that an
  \emph{order dual} $\widetilde{X}$ of a Riesz space $X$ is an
  ordered vector space consisting of all order bounded linear
  functionals on $X$ under the usual algebraic operations, and the
  ordering $f\leq g$ if $\langle f, x \rangle \leq \langle g, x
  \rangle$ for all $x\in X_+$. If $(X, \|\cdot \|)$ is a normed space,
  its \emph{norm dual} $X^\ast$ is the set of continuous linear
  functionals on $X$ equipped with the norm $\|\cdot \|$ defined by
  $\|f\|= \sup \{|\langle f, x\rangle|: x\in X, \|x\|\leq 1\}$.
  If $(X, \|\cdot \|)$ is a locally solid Riesz space, $X^\ast$
  is an ideal in $\widetilde{X}$. If $(X, \|\cdot\|)$ is a Banach
  lattice, then $X^\ast = \widetilde{X}$ and $\|\cdot\|$-topology
  is the finest locally solid topology on $X$. The order dual
  $\widetilde{X}$ of any Riesz space $X$ is an order complete
  Riesz space, and its lattice operations are given by
  \[
  \langle f\vee g ,x \rangle= \sup\{\langle f,y\rangle+ \langle
  g, z
  \rangle: y,z\in X_+ \mbox{ and }
  y+z= x\}
  \]
  and
  \[
  \langle f\wedge g, x\rangle = \inf\{\langle f, y\rangle+ \langle
  g, z \rangle: y,z\in X_+ \mbox{ and } y+z= x\}
  \]
  for all $f, g\in \widetilde{X}$, and $x\in X_+$. These two
  equalities are called the \emph{Riesz-Kantorovich formulas}.
  The \emph{Hahn-Banach Theorem}
  claims that a continuous linear functional defined on a subspace
  of a normed space can be extended to a continuous linear functional
  on the entire space, with its norm preserved. In addition, the
  separation theorem says for any two disjoint (non-empty) convex
  subsets $A$ and $B$ in a Banach lattice $(X, \|\cdot\|)$, if
  either $A$ or $B$ has non-empty interior, then there exists a
  non-zero $f\in X^\ast$ that separates $A$ and $B$.

  Let $X$ be a Banach lattice. If $\Omega$ is a finite set, then
  $X^\Omega$ is endowed with pointwise algebraic operations, the
  pointwise order and the product norm of $X$ is a Banach lattice.
  If $x\in
  (X, \|\cdot\|)^\Omega$ (i.e., $X^\Omega$ equipped with the
  $\|\cdot\|^\Omega$-topology) and $g\in ((X, \|\cdot\|)^\Omega)^\ast$,
  then there is an element $f\in ((X, \|\cdot\|)^\ast)^\Omega$ such that
  $\langle g, x\rangle= \sum_{\omega\in \Omega}\langle f(\omega),
  x(\omega)\rangle$ and vise-versa.
  For any $x\in X$, the \emph{principal ideal} generated
  by $x$ is defined by
  $L(x)= \{y\in X: |y|\leq n |x| \mbox{ for some } n\in \mathbb{N}\}$.
  A point $x\in X_+$ is called an \emph{order unit} of $X$ if $L(x)= X$,
  and a \emph{quasi-interior point} of $X_+$ if $L(x)$ is norm dense
  in $X$ (equivalently, $\langle f, x\rangle> 0$ for all $f \in
  X_+^\ast \setminus \{\bf 0\}$). If $x$ is a quasi-interior point
  of $X_+$, we write $x\gg 0$. An order unit is a quasi-interior
  point, but the converse is not true in general. Note that $L(x)$
  with the norm $\|y\|_{\hat{x}}= \inf\{\lambda>0: |y|< \lambda|x|\}$
  is an \emph{AM-space} with $x$ as an order unit.

  Let $(\Omega, \Sigma, \mu)$ be a measure space and $X$ a Banach
  lattice. A function $\phi:\Omega\to X$ that assumes only a finite
  number of values, say $x_1, ..., x_n$, is called a \emph{simple
  function} if $A_i= \phi^{-1}(\{x_i\})\in \Sigma$ for each $i$. As
  usual, the formula $\phi= \sum_{i= 1}^n x_i {\bf 1}^{A_i}$ is the
  standard representation of $\phi$, where
  \[
  {\bf 1}^{A_i}(t)= \left\{
  \begin{array}{ll}
  1, & \mbox{if $t\in A_i$;}\\[0.5em]
  0, & \mbox{otherwise,}
  \end{array}
  \right.
  \]
  is called the \emph{characteristic function} of $A_i$ on $\Omega$.
  Moreover, we say that a function $f:\Omega\rightarrow X$ is
  \emph{$\mu$-measurable} if there exists a sequence of simple
  functions $\{\phi_n: n\ge 1\}$ such that $\lim \|f(t)-\phi_n(t)\|=
  0$ for almost all $t\in \Omega$. A $\mu$-measurable function
  $f:\Omega\rightarrow X$ is said to be \emph{Bochner integrable}
  if there exists a sequence of simple functions $\{\phi_n: n\ge 1\}$
  such that the real measurable function $\|f(t)- \phi_n(t)\|$ is
  Lebesgue integrable for each $n$ and $\lim\int_\Omega
  \|f(t)- \phi_n(t)\|d\mu(t)= 0$. In this case, for each
  $E\in \Sigma$, the Bochner integral of $f$ over $E$ is defined by
  $\int_E f(t)\mu(t)= \lim\int_E \phi_n(t)d\mu(t)$,
  where the last limit is in the norm topology on $X$.


\begin{thebibliography}{00}

  \bibitem{Aliprantis-Border:05}
  Aliprantis, C.~D., Border, K.~C.: Infinite dimensional analysis: A
  hitchhiker's guide, Third edition, Springer, 2006

  \bibitem{Aliprantis-Tourky:00}
  Aliprantis, C.~D., Tourky, R., Yannelis, N.~C.: Core conditions in
  general equilibrium theory, J Econ Theory {\bf 92}, 96-121 (2000)

  \bibitem{Arrow-Debreu:54}
  Arrow, K.~J., Debreu, G.: Existence of an equilibrium for a competitive
  economy, Econometrica {\bf 22}, 265--290 (1954)

  \bibitem{Aubin:79}
  Aubin, J.~P.: Mathematical methods of game and economic theory,
  North-Holland Publishing Co., Amsterdam-New York, 1979

  \bibitem{Aumann:64}
  Aumann, R.~J.: Markets with a continuum of traders, Econometrica {\bf 32},
  39--50 (1964)

  \bibitem{Debreu:59}
  Debreu, G.: Theory of value: an axiomatic analysis of economic
  equilibrium, John Wiley \& Sons, New York (1959)

  \bibitem{Einy-Moreno-Shitovitz:00}
  Einy, E., Moreno, D., Shitovitz, B.: On the core of an economy with
  differential information, J Econ Theory
  {\bf 94}, 262--270 (2000)

  \bibitem{Einy-Moreno-Shitovitz:01}
  Einy, E., Moreno, D., Shitovitz, B.: Competitive and core allocations
  in large economies with differentiated information, Econ Theory
  {\bf 18}, 321--332 (2001)

  \bibitem{Evern-Husseinov:08}
  Evren, \"{O}., H\"{u}sseinov, F.: Theorems on the core of an economy
  with infinitely many commodities and consumers, J Math Econ
  {\bf 44}, 1180--1196 (2008)

  \bibitem{Graziano-Meo:05}
  Gabriella Graziano, M., Meo, C.: The Aubin private core of
  differential information economies, Decisions in Economics and
  Finance, {\bf 28(1)}, 9--31 (2005).

  \bibitem{García-Cutrín-Hervés-Beloso:93}
  Garc\'{i}a-Cutr\'{i}n, J., Herv\'{e}s-Beloso, C.:
  A discrete approach to continuum economies, Econ Theory {\bf 3},
  577--583 (1993)

  \bibitem{Global:72}
  Grodal, B.: A second remark on the core of an atomless economy,
  Econometrica {\bf 40}, 581--583 (1972)

  \bibitem{Herves-Moreno-Nunez-Pascoa:00}
  Herv\'{e}s-Beloso, C., Moreno-Garc\'{i}a, E.,
  N\'{u}\~{n}ez-Sanz, C, P\'{a}scoa, M.~R.: Blocking efficiency
  of small coalitions in Myopic Economics, J Econ Theory {\bf 93},
  72--86 (2000)

  \bibitem{Herves-Beloso-Moreno-Garcia-Yannelis:05a}
  Herv\'{e}s-Beloso, C., Moreno-Garc\'{i}a, E., Yannelis, N.~C.:
  An equivalence theorem for differential information economy,
  J Math Econ {\bf 41}, 844--856 (2005a)

  \bibitem{Herves-Beloso-Moreno-Garcia-Yannelis:05b}
  Herv\'{e}s-Beloso, C., Moreno-Garc\'{i}a, E., Yannelis, N.~C.:
  Characterization and incentive compatibility of Walrasian expectations
  equilibrium in infinite dimensional commodity spaces,
  Econ Theory {\bf 26}, 361--381 (2005b)

  \bibitem{Hildenbrand:74}
  Hildenbrand, W.: Core and equilibria in large economies, Priceton
  University Press, 1974

  \bibitem{McKenzie:59}
  McKenzie, L.W.: On the existence of general equilibrium for a competitive
  market, Econometrica {\bf 27}, 54--71 (1959)

  \bibitem{Podczeck:96}
  Podczeck, K.: Equilibria in vector lattices without ordered preferences
  or uniform properness, J Math Econ {\bf 25}, 465--485 (1996)

  \bibitem{Podczeck:03}
  Podczeck, K.: Core and Walrasian Equilibria when agents' characteristics
  are extremely dispersed, Econ Theory {\bf 22}, 699--725 (2003)

 \bibitem{Podczeck:08}
  Podczeck, K., Yannelis, N.~C.: Equilibrium theory with asymmetric
  information and with infinitely many commodities, J Econ Theory {\bf 141},
  152--183 (2008)

  \bibitem{Radner:68}
  Radner, R.: Competitive equilibrium under uncertainty, Econometrica
  {\bf 36}, 31--58 (1968)

  \bibitem{Radner:82}
  Radner, R.: Equilibrium under uncertainty, pp. 923--1006 in Handbook
  of Mathematical Economics Vol 2, North Holland, Amsterdam (1982)

  \bibitem{Schmeidler:72}
  Schmeidler, D.: A remark on the core of an atomless economy,
  Econometrica {\bf 40}, 579--580 (1972)

  \bibitem{tourky:01}
  Tourky, R., Yannelis, N. C.: Markets with many more agents than
  commodities: Aumann's ``hidden" assumption, J Econ Theory {\bf 101},
  189-221 (2001)

  \bibitem{uhl:69}
  Uhl, J. J., Jr.: The range of a vector valued measure, Proc. Amer.
  Math. Soc. {\bf 23}, 158-163 (1969)


  \bibitem{Vind:72}
  Vind, K.: A third remark on the core of an atomless economy,
  Econometrica {\bf 40}, 585--586 (1972)

  \bibitem{Yannelis:91}
  Yannelis, N.~C.: The core of an economy with differential information,
  Econ Theory {\bf 1}, 183--197 (1991)

  \end{thebibliography}
   \end{document}